\documentclass[12pt,a4paper]{amsart}
\usepackage{amsfonts}
\usepackage{amssymb}
\usepackage{bbm}
\usepackage{comment}
\usepackage{enumitem}
\usepackage{dsfont}
\usepackage{amsthm}
\usepackage{amsmath}
\usepackage{commath}
\usepackage{url}
\usepackage{epsfig}
\usepackage{subfigure}
\usepackage{adjustbox}
\usepackage[ruled,vlined,linesnumbered]{algorithm2e}
\usepackage{graphicx}
\usepackage{a4wide}
\usepackage{todonotes}
\usepackage{mathtools}

\makeatletter
\@namedef{subjclassname@2010}{%
  \textup{2010} Mathematics Subject Classification}
\makeatother

\usepackage[T1]{fontenc}
\definecolor{green}{rgb}{0.1,1,0.1}

\newtheorem{corollary}{Corollary}[section]
\newtheorem{theorem}[corollary]{Theorem}
\newtheorem{lemma}[corollary]{Lemma}
\newtheorem{proposition}[corollary]{Proposition}
 \newtheorem{lem}[corollary]{Lemma}
\newtheorem*{rmk}{Remark}
\theoremstyle{definition}

\newcommand{\ls}{\leqslant}
\newcommand{\gs}{\geqslant}

\newcommand{\one}{\mathds{1}}
\newcommand{\xn}{(x_n)_{n \in \mathbb{N}} }

\newcommand{\om}{\omega}
\def \D {\Delta}
\def \bE {\mathbb E}
\def \bN {\mathbb N}

\def \bR {\mathbb R}

\newcommand{\cP}{\mathcal{P}}

\newcommand{\bfi}{\mathbf{i}}

\begin{document}
\baselineskip=17pt
%%%%%%%%%%%%%%%%
\title{  Poissonian Correlations of Higher Orders }
\author{Manuel Hauke}
\address{TU Graz, Austria}
\email{hauke@math.tugraz.at}
\author[A. Zafeiropoulos]{Agamemnon Zafeiropoulos}
\address{NTNU Trondheim, Norway}
\email{agamemnon.zafeiropoulos@ntnu.no}

\date{}
\thanks{ %insert grant acknowledgements here
	AZ is supported by a postdoctoral fellowship funded by Grant 275113 of the Research Council of Norway} 
\begin{abstract}
 We show that any sequence $\xn \subseteq [0,1]$ that has Poissonian correlations of $k$\,--\,th order is  uniformly distributed,  also providing a quantitative description of this phenomenon. 
 %in terms of the asymptotic size of an appropriately defined average of the correlation function.
 Additionally, we extend connections between metric correlations and additive energy, already known for pair correlations, to higher orders.
%  We briefly examine the connection of metric correlations of order $3$ with the notion of additive energy. 
Furthermore, we examine how the property of Poissonian $k$\,--\,th correlations is reflected in the asymptotic size of the moments of the function $F(t,s,N) = \#\{n\ls N : \|x_n - t\| \ls s/(2N) \},\, t\in [0,1]. $
\end{abstract}

%\keywords{aaaa, bbbb, cccc}
\maketitle

\section{Introduction}

\subsection{The main results} Let $\xn\subseteq [0,1]$ be a sequence and $k\gs 2$ be an integer. Given a  compactly supported  test function  $f :\mathbb{R}^{k-1}\to [0,\infty)$ we define the {\it $k$\,--\,th order correlation function} of the sequence $\xn$ with respect to $f$ to be  
 \begin{equation} \label{kthcorrelationfunction}
 R_k(f,N) = \frac{1}{N}\sum_{\substack{i_1,\ldots,i_k \ls N\\ \text{ distinct}}}\hspace{-2mm} f\left(N(\!(x_{i_1}-x_{i_2})\!),N(\!(x_{i_1}-x_{i_3})\!),\ldots,N(\!(x_{i_1}-x_{i_k})\!)\right).
\end{equation}
Here $(\!(x)\!)$ denotes the signed distance of $x$ from the origin modulo $1$ (see Section \ref{notation} for a proper definition).

We say that the sequence $\xn$ has \textit{Poissonian $k$\,--\,th order correlations} if 
\begin{equation}\label{poisson}\lim_{N \to \infty}R_k(f,N) =  \int_{\mathbb{R}^{k-1}}f(x)\,\mathrm{d}x  \qquad \text{for any } f\in C_{c}(\mathbb{R}^{k-1}).  \end{equation}
A discussion on equivalent definitions of Poissonian correlations appearing in the literature can be found in Appendix A. There we explain that a sequence has Poissonian $k$\,--\,th order correlations if and only if \eqref{poisson} holds for any $f$ which is the characteristic function of some rectangle in $\mathbb{R}^{k-1}.$ \par 
Having Poissonian $k$\,--\,th order correlations can be viewed as a pseudo-randomness property of the sequence $\xn$ in the following sense: when $(Y_n)_{n\in\mathbb{N}}$ is a sequence of independent, uniformly distributed random variables in $[0,1]$, then almost surely, the sequence $(Y_n(\omega))_{n\in\mathbb{N}}$ has Poissonian $k$\,--\,th order correlations (see Appendix B for a proof).
\par
The $k$\,--\,th order correlations of a sequence are a local asymptotic statistics of the gaps of a sequence. Another closely related statistics is the asymptotic gap distribution of the sequence \cite{kr}. It is known that when a sequence has Poissonian correlations of all orders $k\gs2,$ then the asymptotic distribution of its gaps is also Poissonian; a proof can be found in \cite[Appendix A]{kr}. \par
The term \emph{Poissonian} comes from the fact that \eqref{poisson} is in accordance with the almost sure statistical behavior of gaps between random points coming from a Poisson process. Originally, the motivation for studying the gap statistics of point sequences came from theoretical physics, where the Berry\,--\,Tabor conjecture predicts that the spacings of the energy eigenvalues of generic integrable quantum systems follow the Poissonian model (see \cite{marklof2} for a survey in mathematical language). For some quantum systems the sequence of energy eigenvalues follows a simple arithmetic formula, but establishing the correlations to be Poissonian is usually a very substantial challenge (that often becomes more and more difficult as the order of the correlations increases). For some contributions concerning sequences of cognizable physical origin see for example \cite{heathbrown,my,sarnak}. However, the subject has also gained significant interest on a purely mathematical level, where the correlations of general sequences of arithmetic origin were studied; see for example \cite{all, bw,rs}. Most results only concern the case of correlations of order $k=2$ (known as \emph{pair correlations}), while correlations of higher order are combinatorially and analytically more difficult to study and often out of reach; among the relatively few results in that direction are \cite{rz} and \cite{ty}.

\par 
A topic of particular interest has been the connection of Poissonian pair correlations with uniform distribution properties. Recall that a sequence $\xn\subseteq [0,1]$ is called {\it uniformly distributed} if for any $0\ls a< b\ls 1$ we have 
\[ \lim_{N \to \infty} \frac{1}{N}\#\{  n\ls N: a\ls x_n \ls b\} = b-a.  \] 
To be more specific, it has been shown that when a sequence has Poissonian pair correlations it is also uniformly distributed. \\                 

\noindent {\bf Theorem A. } {\it Let $\xn\subseteq [0,1]$ be a sequence. If $\xn$ has Poissonian pair correlations, then $\xn$ is uniformly distributed. } \\

Theorem A was proved independently by Aistleitner, Lachmann and Pausinger \cite{alp} and by Grepstad and Larcher \cite{sigrid}. Additional proofs were given later by Steinerberger in \cite{steinerberger}
and, in a much more general setup, by Marklof \cite{marklof}. These four proofs are all essentially different. \par 
The authors of \cite{alp} also prove a quantitative version of Theorem A that is interesting in its own right. Before we present this  version of Theorem A, recall that a function $G:[0,1]\to \mathbb{R}$ is called   {\it an asymptotic distribution function} of a sequence $\xn \subseteq [0,1]$ if there exists a strictly increasing sequence of integers $(N_j)_{j\in\bN}$ such that 
\begin{equation} \label{G_definition} 
G(x) = \lim_{j \to \infty}\frac{1}{N_j}\#\{n \ls N_j : 0\ls x_n \ls x \}, \qquad 0\ls x \ls 1.   
\end{equation}
 By the Helly selection principle, every sequence $\xn \subseteq [0,1]$ has at least one asymptotic distribution function (see e.g. \cite[Ch. 1, Thm. 7.1]{kuipers}). Whenever \mbox{$G:[0,1]\to \bR$} is the unique asymptotic distribution function of the sequence $\xn$, that is, whenever
 \[ G(x) = \lim_{N \to \infty}\frac{1}{N }\#\{n \ls N  : 0\ls x_n \ls x \}    \]
holds for all $x\in [0,1]$, we will simply refer to it as the asymptotic distribution function of $\xn$. \par The quantitative version of Theorem A is the following statement. 
\vspace{2mm}

\noindent{\bf Theorem B. } {\it Assume that the sequence $\xn\subseteq [0,1]$ has the unique asymptotic distribution function $G:[0,1] \to \mathbb{R}.$ Assume also that there is a function \mbox{$F:[0,\infty)\rightarrow [0,\infty]$} such that 
\[ F(s) = \lim_{N \to \infty}\frac{1}{N} \# \Big\{k\neq \ell \ls N : \|x_k - x_\ell\| \ls \frac{s}{N}  \Big\}, \qquad s>0. \]
%where $\|.\|$ denotes the distance to the next integer.
Then the following hold: 
\begin{itemize}
	\item[(i)] If $G$ is not absolutely continuous, then $F(s)=\infty$ for all $s>0$.
	\item[(ii)] If $G$ is absolutely continuous, then 
	\begin{equation*}\label{thm_b} \limsup_{s \to \infty}\frac{F(s)}{2s} \gs \int_0^1 g(x)^2\,\mathrm{d}x \end{equation*}
where $g$ is the density function of the corresponding measure (that is, $g = G'$ almost everywhere).
\end{itemize}}
 As a main result of this paper, we prove that Theorem A can be generalised to sequences with Poissonian correlations of order $k\gs 2$: having Poissonian correlations of any order $k\gs 2$ is a stronger property than uniform distribution.
 \begin{theorem}\label{main}
If the sequence $(x_n)_{n \in \mathbb{N}}\subseteq [0,1]$ has Poissonian $k$\,--\,correlations for some $k \gs 2$, then it is uniformly distributed.
\end{theorem}

In fact, with a little more effort than in the proof of Theorem \ref{main} we are able to prove a stronger result that can be viewed as a generalisation of Theorem B in the context of $k$\,--\,th order correlations.  In what follows, for $s>0$ we shall write $R_k(s,N) = R_k(f_s,N)$ where $f_s$ is the test function $f_s = \mathds{1}_{[-s,s]^{k-1}}$.
\begin{theorem}\label{main2}
        Let $(x_n)_{n \in \mathbb{N}} \subseteq [0,1]$, $G:[0,1]\to\bR $ be an asymptotic distribution function of $\xn$ and $(N_j)_{j\in\bN}$ be a sequence as in \eqref{G_definition}. Then the following hold:
\begin{itemize}
	\item[(i)] If $G$ is not absolutely continuous, then $\lim\limits_{j\to\infty}R_k(s,N_j)=\infty$ for all $s>0$.
	\item[(ii)] If $G$ is absolutely continuous, then  
	\begin{equation}\label{k_thmoment}\limsup_{s \to \infty}\frac{\limsup\limits_{j\to\infty} R_k(s, N_j)}{(2s)^{k-1}} \gs \int_{0}^{1} g(x)^k \,\mathrm{d}x, \end{equation}
where $g$ is the density function of the corresponding measure. 
\end{itemize}
 \end{theorem}
 To see why Theorem \ref{main2} is indeed a stronger version of Theorem \ref{main}, observe that if $\xn$ has Poissonian $k$\,--\,th correlations but is not uniformly distributed, it will have an asymptotic distribution function $G$ whose density function $g$  is not constantly equal to $1$ (the existence of $g$ follows from (i) of Theorem \ref{main2}). Therefore $\int_0^1g(x)^k\mathrm{d}x>1$ and \eqref{k_thmoment} leads to a contradiction. \par 
The reader might spot two subtle differences between Theorem B and its generalisation, Theorem \ref{main2}. First, in Theorem \ref{main2} we assume that $G$ is an asymptotic distribution function of $\xn$, not necessarily unique as in Theorem B. Second, we do not require that the limit $\lim_{N \to \infty} R_k(s,N)$ exists, but instead, we work with the term $\limsup_{j \to \infty} R_k(s,N_j)$. The additional assumptions  in Theorem B are not essential, and the proof in \cite{alp} can be easily modified under the slightly weaker hypotheses of Theorem \ref{main2}. These minor modifications also make Theorem B an actually stronger result than Theorem A. \\

In the proof of both Theorems \ref{main} and \ref{main2}, we shall make  use of several variants of the correlation function $R_k(f,N)$ that was defined in \eqref{kthcorrelationfunction}. To be more specific, given some scales $s_1,s_2,\ldots, s_{k-1}>0$ we define the correlation function $R_k(s_1,\ldots,s_{k-1},N)$ by 
\begin{equation} \label{kthcorrelationfunction_old_def}
R_k(s_1,\ldots,s_{k-1},N) = \frac{1}{N}\#\left\{ \parbox{7em}{$ i_1,\ldots,i_k \leqslant N $ \\  $i_j \neq i_\ell \,\forall j \neq \ell$}\hspace{-5mm} : 
 \,  \|x_{i_1} - x_{i_{r+1}}\|\ls \frac{s_r}{N}\,\,\, (1\ls r <k)  \right\}  
\end{equation}
(where $\|x\|$ denotes the distance of $x\in\mathbb{R}$ to its nearest integer). Since we can write $R_k(s_1,\ldots,s_{k-1},N)$  as $R_k(\one_{B},N)$ where $B$ is the rectangle $[-s_1,s_1]\times \ldots \times [-s_{k-1},s_{k-1}],$  it follows from \eqref{poisson} and an approximation argument  that sequences with Poissonian $k$\,--\,th correlations  also satisfy
 \begin{equation*}  
     \lim_{N\to\infty}  R_k(s_1,\ldots,s_{k-1},N) = (2s_1) \cdots (2s_{k-1}) \quad \text{ for all } s_1,\ldots, s_{k-1}>0 
 \end{equation*}
(see also Appendix A for a discussion). \par Furthermore, we define the correlation function $R_k^*(s_1,\ldots,s_{k-1},N)$ by
\begin{align} \label{rkstardef}
R_k^*(s_1,\ldots, s_{k-1},N) = \frac{1}{N}\#\left\{ i_1,\ldots,i_k \ls N:
		 \|x_{i_1} - x_{i_{r+1}}\|\ls \frac{s_r}{N}\,\, (1\ls r<k) \right\}  .
		\end{align}
That is, in the definition of $R_k^*(s_1,\ldots,s_{k-1},N)$ we allow indices to be equal. We will also make use of the appropriate averages of $R_k$ and $R_k^*$ defined as   
\begin{equation*}
C_k(s_1,\ldots,s_{k-1},N) = \mathop{\iint\ldots\int}_{B(s_1,\ldots,s_{k-1})}\hspace{-1mm}R_k(\sigma_1,\ldots,\sigma_{k-1},N)\, \mathrm{d}\sigma_1 \mathrm{d}\sigma_2\ldots \mathrm{d}\sigma_{k-1} \end{equation*}
and
\begin{equation*} 
C_k^*(s_1,\ldots,s_{k-1},N) = \mathop{\iint\ldots\int}_{B(s_1,\ldots,s_{k-1})}\hspace{-1mm}R_k^*(\sigma_1,\ldots,\sigma_{k-1},N)\, \mathrm{d}\sigma_1 \mathrm{d}\sigma_2\ldots \mathrm{d}\sigma_{k-1} , \end{equation*}
where $B(s_1,\ldots,s_{k-1})$ denotes the rectangle $[0,s_1]\times [0,s_2] \times \ldots \times [0,s_{k-1}].$
Finally, when the scales $s_1,\ldots, s_{k-1}$ are all equal to $s>0$ we write for simplicity 
\begin{equation*} \label{all_inclusive} \begin{gathered}
R_k(s,N)   = R_k(s,\ldots,s,N),\qquad R_k^*(s,N)  =  R_k^*(s,\ldots,s,N), \\
C_k(s,N)   = C_k(s,\ldots,s,N), \qquad C_k^*(s,N)  =  C_k^*(s,\ldots,s,N). \end{gathered} \end{equation*}
We note that the definition of $R_k(s,N)$ here agrees with the definition given right before Theorem \ref{main2}.

\subsection{Some consequences of the main results} In the course of the proofs of the stated theorems, we examine the relations between the correlation functions of different orders for a given sequence. These relations enable us to prove that when a sequence $\xn$ fails to have Poissonian $k$\,--\,th order correlations in a very strong sense, then it will also not have Poissonian correlations of any higher order.

\begin{theorem} \label{thm3}
Let $\xn\subseteq [0,1]$ be a sequence and assume that for some scales $s_1,\ldots, s_{k-1}>0$ the $k$\,--\,th order correlation function $R_k(s_1,\ldots,s_{k-1},N)$ satisfies
\[\limsup_{N\to\infty}R_k(s_1,\ldots,s_{k-1},N) = \infty.  \]
Then the sequence $\xn$ does not have Poissonian $p$\,--\,th order correlations for any $p\gs k.$
\end{theorem}
Theorem \ref{thm3} is one of the key ingredients we will use to derive new results on the {\it metric theory} of Poissonian correlations:  an increasing sequence $\mathcal{A}=(a_n)_{n\in\bN}\subseteq \bN$ is fixed, and we study the Lebesgue measure of $x\in [0,1]$ such that the sequence $(a_n x)_{n\in\bN}$  has Poissonian correlations.  As with the correlations of fixed sequences, most of the metric results  to date are concerned with the $k=2$ case. \par In the metric setup, the authors of \cite{all} have established several statements on the connection of Poissonian correlations with the notion of {\it additive energy}.  Recall that the additive energy of a finite set $A$ is defined as  
\[E(A) = \#\{a,b,c,d\in A : a+b = c+d \}  \]
(see \cite[Chapter 2]{taovu} for more details). Writing $\mathcal{A}_N = (a_n)_{n\ls N}$ for the set of the first $N$ elements of $\mathcal{A},$ Aistleitner, Larcher \& Lewko \cite{all}
proved that whenever $E(\mathcal{A}_N)= \mathcal{O}(N^{3-\varepsilon}), N\to\infty $ for some $\varepsilon>0,$ then the sequence $(a_nx)_{n\in\bN}$ has Poissonian pair correlations for almost all $x\in[0,1].$ Examining to what extent this bound on the additive energy is optimal, Bourgain \cite[Appendix]{all} showed that if $E(\mathcal{A}_N) = \Omega(N^3), N\to\infty$ then $(a_nx)_{n\in\bN}$ does not have Poissonian pair correlations for all $x$ in a subset of $[0,1]$ of positive Lebesgue measure, while on the other hand there exists a sequence $\mathcal{A} = (a_n) \subseteq \bN$ with $E(\mathcal{A}_N) = o(N^3), N\to\infty$ such that $(a_n x)_{n\in\bN}$ does not have Poissonian pair correlations for almost all $x\in [0,1].$ \par In the present paper, we are able to deduce an analogue of the first\,--\,mentioned result of Bourgain, proving that when the additive energy of $\mathcal{A}$ is of maximal order of magnitude, then the sequence $(a_n x)_{n\in\bN}$ fails to have Poissonian triple correlations (i.e. of order $k=3$) for Lebesgue almost all $x$. Further, as mentioned previously, we use Theorem \ref{thm3} to generalise the second\,--\,mentioned result of Bourgain for $k$\,--\,th order correlations.

\begin{theorem}\label{ap_thm}(i) There exists a set $\mathcal{A} = (a_n)_{n=1}^{\infty} \subseteq \mathbb{N}$ with additive energy $E(\mathcal{A}_N) = o(N^3), N \to \infty$ such that for Lebesgue almost all $x\in [0,1]$ the sequence $(a_nx)_{n=1}^{\infty}$ does not have Poissonian correlations of any order $k\gs 2.$ \par
(ii) Let $\mathcal{A} = (a_n x)_{n\in\bN} \subseteq \bN$ be a sequence such that $E(\mathcal{A}_N) = \Omega(N^3)$ as  $N\to\infty$. Then for all $x$ in a set of positive Lebesgue measure, $ (a_n x)_{n\in\bN}$ does not have Poissonian  triple correlations. 
\end{theorem}

\subsection{Poissonian Correlations and the number of points in small intervals}

As a final result of this paper, we seek to exhibit a connection of the property of Poissonian k\,--\,th correlations with the number of elements of a sequence in sufficiently small intervals. More formally, given $s>0$ and $N\gs 1$ we define 
\begin{equation} \label{Fdef}
    F(t,s,N) = \#\Big\{n\ls N : \| x_n -t \| \ls \frac{s}{2N} \Big\}, \qquad 0\ls t\ls 1. 
\end{equation}

Heuristically, if $t$ is seen as a random variable uniformly distributed in $[0,1]$, then $F(t)=F(t,s,N)$ can be viewed as  the number of points of the sequence $\xn$ in a random interval of length $s/N$. \par Our purpose is to establish a link between the property of Poissonian $k$\,--\,th correlations and the asymptotic size of the $k$\,--\,th moment of $F(t,s,N)$.  For the $k = 2$ case, it is already known that a sequence $\xn$ has Poissonian pair correlations if and only if 
\[\lim_{N \to \infty} \int_{0}^{1} F(t,s,N)^2\,\mathrm{d}t = s^2 +s \qquad \text{ for all } s>0. \]
This is shown in \cite{marklof3} and also implicitly in \cite[Thm 3(i)]{heathbrown}. Regarding correlations of higher orders, some relevant results are shown in \cite{technau_walker} for the triple correlations of sequences of the form $(n^2\alpha)_{n\in\bN}$ and for values of the length $s>0$ that lie in a range that depends on $N$. \par We hereby consider both the $k$\,--\,th moment of $F(t,s,N)$ as well as its $k$\,--\,th factorial moment. To be more specific, given $s>0, N\gs 1$ and $k\gs 2$ we set
  \begin{equation}\label{moments}  \begin{aligned}I_k(s,N) & = \int_0^1 F(t,s,N)(F(t,s,N)-1) \cdots (F(t,s,N)-(k-1)) \,\mathrm{d}t,\\
    I_k^{*}(s,N) & = \int_{0}^{1} F(t,s,N)^k \,\mathrm{d}t.\end{aligned} \end{equation}
We prove that the property of Poissonian correlations of $k$\,--\,th order is reflected in the asymptotic behaviour of $I_k(s,N)$ and $I_k^*(s,N).$
\begin{theorem}\label{connection_to_higher_moments}
    Let $k \gs 2$ and $I_k(s,N), I_k^*(s,N)$ be as in \eqref{moments}. Assume the sequence $\xn$ has Poissonian $k$\,--\,th correlations. Then the following statements hold: \\
(i) $\lim\limits_{N \to \infty} I_k(s,N) = s^k $ for all $s>0.$ \vspace{1mm} \newline 
(ii) $\limsup\limits_{N \to \infty} I_k^{*}(s,N) = s^k + \mathcal{O}_k(s^{k-1}), \quad s \to \infty$. \newline 
(iii) If, in addition, $\xn$ has Poissonian $\ell$\,--\,correlations for all $\ell \ls k$, then 
    \begin{equation}\label{limit_of_moment} \lim_{N \to \infty} I_k^{*}(s,N) = s^k + c_{k,k-1}s^{k-1} + \ldots +c_{k,1}s , \end{equation}
     where $c_{k,i}, i = 1,\ldots,k-1$ denote the Stirling numbers of the second kind.
 \end{theorem}
Theorem \ref{connection_to_higher_moments} provides further evidence for the connection between sequences with Poissonian correlations on the one hand, and random variables that follow the Poisson distribution on the other. It is known \cite{riordan} that the $k$\,--\,th factorial moment of a random variable following the Poisson distribution with parameter $s>0$ is equal to $s^k$, while its $k$\,--\,th moment is equal to the polynomial on the right\,--\,hand side of \eqref{limit_of_moment} (this is called the Bell polynomial of degree $k,$ see also \cite{carlitz}).

\subsection{Notation}\label{notation} Given two functions $f,g:(0,\infty)\to \mathbb{R},$ we write $f(t) = \mathcal{O}(g(t)),  t\to\infty,$ $f(t)= o(g(t)), t\to\infty$ and $f(t)=\Omega(g(t)), t\to\infty$ when  \vspace{-2mm}
\[\limsup_{t\to\infty} \frac{|f(t)|}{|g(t)|} < \infty, \quad   \quad \lim_{t\to\infty} \frac{f(t)}{g(t)} =0 \quad \text{ or } \quad \limsup_{t\to\infty}\frac{f(t)}{g(t)}>0 \] respectively. Any  dependence of the value of the limsup above on potential parameters is denoted by the appropriate subscripts in the $\mathcal{O}$\,--\,symbol. Given a real number $x\in \bR,$ we write $\{x\}$ for the fractional part of $x$, $\|x\|=\min\{|x-k|: k\in\mathbb{N}\}$ for the distance of $x$ from its nearest integer, and $$(\!(x)\!)=\begin{cases} \{x\}, &\text{ if } 0\ls \{x\} \ls \tfrac{1}{2} \\ \{x\}-1, &\text{ if } \tfrac{1}{2}< \{x\} < 1  \end{cases} $$ for the signed distance of $x$ from the origin modulo $1$. Further, we use the symbol $\{\,\cdot\,\}^+$ for the function \vspace{-2mm}
\[ \{x\}^+ = \begin{cases}x, &\text{ if } x\gs 0 \\ 0, &\text{ if } x<0 . \end{cases} \]
Throughout the paper, we shall implicitly consider the unit interval $[0,1]$ equipped with the topology induced by $\norm{\,\cdot \, }$ because we deal with distribution of sequences modulo $1$. This is homeomorphic to the interval $[0,1)$ with the same topology, so  for convenience, we will  work interchangingly with $[0,1]$ and $[0,1)$. \par We use the standard notation $e(x)=e^{2
\pi ix}.$  We also write $B(x_0,r)= \{x\in [0,1] : \|x- x_0\| \ls r\} $ for the interval with center $x_0$ and length $2r$ modulo $1$. The characteristic function of a set $A$ is denoted by $\one_A.$

\subsection{ Directions for Further Research} We end this introductory part with some interesting questions that would shed more light on the properties of sequences with Poissonian correlations of $k$\,--\,th order. \par
$\bullet$ In Theorem \ref{main} we proved that when a sequence $\xn$ has Poissonian $k$\,--\,th correlations it is uniformly distributed, but we do not know whether the correlations of orders $m<k$ also follow the Poissonian model. Are Poissonian correlations of order $k+1$ a property stronger than Poissonian correlations of order $k$? In other words, does any sequence $\xn$ with Poissonian $k$\,--\,th correlations also have Poissonian correlations of all orders $2\ls m<k?$ \par $\bullet$ We would like to know if some partial converse to Theorem \ref{connection_to_higher_moments} is true. Is it true, for example, that whenever \eqref{limit_of_moment} holds, the sequence $\xn$ has Poissonian correlations of all orders up to $k$?

\section{Properties of the functions $R_k, R_k^*$.}

In the present section we prove several properties of the functions $R_k$ and  $R_k^*$  defined in the introduction that will be used later in the proof of the main results. \\

\par We start by proving the inequality that will be the key ingredient in the proof of Proposition \ref{HaukePrinciple}. 
 
\begin{lem}
Let $m\gs 1.$ For any $M \gs 1$ and for all non-negative real numbers $x_1,x_2,\ldots, x_M \gs 0$ we have 
\begin{equation}\label{holder_type}
(x_1^{m+1}+ x_2^{m+1} + \ldots + x_{M}^{m+1}) \gs \frac{1}{M}(x_1 + x_2 + \ldots + x_M)(x_1^m + x_2^m + \ldots + x_M^m).
	\end{equation}
\end{lem}
\begin{proof} Applying the Hölder inequality with exponents $p=m+1$ and $q=(m+1)/m$ to the $M$--tuples $(x_1,\ldots,x_M)$ and $(1,\ldots,1)$ we get
\begin{equation}\label{hölder1}
x_1 + x_2 + \ldots +x_M  \ls 
\left(x_1^{m+1} + \ldots + x_M^{m+1}\right)^{\frac{1}{m+1}}M^{\frac{m}{m+1}},
\end{equation}
while the Hölder inequality with the same exponents  applied to the $M$--tuples $(1,\ldots,1)$ and $(x_1^m,\ldots,x_M^m)$   yields
\begin{equation}\label{hölder2} x_1^m + x_2^m + \ldots + x_M^m \ls \left(x_1^{m+1} + \ldots + x_M^{m+1}\right)^{\frac{m}{m+1}}
   M^{\frac{1}{m+1}}.
\end{equation}
Multiplying \eqref{hölder1} and \eqref{hölder2}, we obtain \eqref{holder_type}.\end{proof}
 
The following proposition  provides a relation between correlation functions of different orders. This result will later have a key role in the proof of both Theorems \ref{main} and \ref{main2}, while it straightforwardly implies Theorem \ref{thm3}. 

\begin{proposition}\label{HaukePrinciple}
Let $\xn \subseteq [0,1]$ be a sequence and $m\gs 2$. There exists a constant $s_m > 0$ such that for any $s > s_m$, the inequality 
	\begin{equation} \label{conclusion}
	R_m\Big(\frac{s}{3},N\Big) \ls \dfrac{6}{s} R_{m+1}(s,N)  \end{equation}
	holds for all $N\gs N_0(s,m).$ Moreover, the values of the constants $s_m$ and $N_0$ are independent of the sequence $\xn$.
\end{proposition}
\begin{proof}
We partition the unit interval into pieces of size approximately $s/N$ and count points in each interval: we set 
	$$ K = K(s,N) = \left\lceil \frac{N}{s} \right\rceil $$
	and for $0 \ls \ell \ls K-1$ we define \begin{equation*}
	y_\ell = y_\ell(s,N) = \#\Big\{1 \ls i \ls N: x_i \in \left[\tfrac{\ell s}{N},\tfrac{(\ell+1)s}{N}\right)\cap [0,1] \Big\} .
	\end{equation*}
Observe that 
	\begin{eqnarray*}
	R_{m+1}(s,N) &\gs& \frac{1}{N}\sum_{\ell=0}^{K-1} y_\ell(y_\ell-1)\ldots (y_\ell-m)\nonumber \\
	&=& \frac{1}{N}\sum_{\ell=0}^{K-1}\Big( y_\ell^{m+1} - c_{m} y_\ell^{m} +  c_{m-1} y_\ell^{m-1} - \ldots + (-1)^{m} c_1 y_\ell\Big) \label{eleven}
\end{eqnarray*}
with $c_i \in \mathbb{N}$.

	First, we consider the case when $m$ is odd.
Since $y_0 + y_1 + \ldots + y_{K-1}=N,$	applying inequality \eqref{holder_type} to \eqref{eleven} we obtain 
	\[R_{m+1}(s,N) \gs \frac{1}{N}\sum_{\ell = 0}^{K-1}[(\tfrac{N}{K}-c_{m})y_\ell^{m} + (\tfrac{N}{K}c_{m-1}-c_{m-2})y_\ell^{m-2} + \ldots + (\tfrac{N}{K}c_2 - c_1)y_\ell].\]
	Note that for any $\varepsilon > 0$, we find that for $N$ sufficiently large 
$\dfrac{N}{K} \gs (1-\varepsilon)s $ and therefore   
	\[R_{m+1}(s,N) \gs \frac{1}{N}\sum_{\ell = 0}^{K-1}[((1-\varepsilon)s-c_{m})y_\ell^{m} + ((1-\varepsilon)sc_{m-1}-c_{m-2})y_\ell^{m-2} + \ldots + ((1-\varepsilon)sc_2 - c_1)y_\ell].\]
	Similarly, for $m$ even, we have
	\[R_{m+1}(s,N) \gs \frac{1}{N}\sum_{\ell =0}^{K-1}[((1-\varepsilon)s-c_{m})y_\ell^{m} + ((1-\varepsilon)sc_{m-1}-c_{m-2})y_\ell^{m-2} + \ldots + ((1-\varepsilon)sc_3 - c_2)y_\ell^2 + c_1 y_\ell].\]
    Now let $\varepsilon>0$ be small enough such that $s(1-\varepsilon) > s_m := 2\max\{c_i: 1 \ls i \ls m\}$. Then all non\,--\,leading terms are positive and hence we can estimate
	\[R_{m+1}(s,N) \gs \frac{1}{N}\sum_{\ell =0}^{K-1}((1-\varepsilon)s-c_{m})y_\ell^{m} \gs \frac{s}{2}\cdot \frac{1}{N}\sum_{\ell=0}^{K-1}y_\ell^{m}, \]
	which implies that
	\begin{equation}\frac{1}{N}\sum_{\ell=0}^{K-1} y_l^{m} \ls \frac{2R_{m+1}(s,N)}{s} \cdot \label{lower_bound}\end{equation}
	We now seek an upper bound for $R_{m}(\frac{s}{3},N)$. For each $0\ls \ell \ls K-1$ consider the  sets
	\begin{equation*}
	    A_\ell = \left[\frac{s\ell}{N},\frac{s(\ell+1)}{N}\right),\qquad A_\ell' = A_{\ell} + \frac{s}{3N}, \qquad   A_\ell'' = A_{\ell} + \frac{2s}{3N} \,  
	\end{equation*}
	(where we understand the intervals modulo 1). Here we have essentially defined three different partitions of the unit interval: the partition $(A_{\ell})_{\ell=0}^{K-1}$ we employed previously, and the partitions $(A_{\ell}')_{\ell=0}^{K-1}$ and $(A_{\ell}'')_{\ell=0}^{K-1}$ that we get by shifting the intervals of the first partition by $s/3N$ and $2s/3N$ respectively. 	Writing $y_\ell',y_\ell''$ for the number of points $x_i$ in $A_{\ell}'$ and $A_{\ell}''$ respectively, it is straightforward to show that an analogue of  \eqref{lower_bound} holds for $(y_\ell')_{\ell=0}^{K-1}$ and $(y_\ell'')_{\ell=0}^{K-1}$. \par At this point, we need to employ the following Lemma, the proof of which we postpone for later in the text. 
\begin{lem}\label{claim} If the $m$\,--\,tuple $(x_{i_1},\ldots,x_{i_{m}})$ contributes something to the sum
	$R_{m}(\frac{s}{3},N)$, then there exists some $0\ls \ell \ls K-1$ such that the points $x_{i_1}, x_{i_2}, \ldots, x_{i_m}$ belong all to $A_{\ell}$ or to $A_{\ell}'$ or to $A_{\ell}''$.  \end{lem}
Using Lemma \ref{claim} we finally obtain
	\begin{align*}R_{m}\Big(\frac{s}{3},N\Big) \ls & \frac{1}{N}\sum_{\ell=0}^{K-1} y_\ell(y_\ell-1)\ldots(y_\ell-(m-1))
	 +  \frac{1}{N}\sum_{\ell=0}^{K-1}y_\ell'(y_\ell'-1)\ldots(y_\ell'-(m-1)) 
	\\&+  \frac{1}{N}\sum_{\ell=0}^{K-1} y_\ell''(y_\ell''-1)\ldots(y_\ell''-(m-1)) 
	 \\ \ls & \frac{1}{N}\sum_{\ell =0}^{K-1}\big(y_\ell^m +  (y_{\ell}')^m +  (y_{\ell}'')^m\big) \ls \frac{6}{s}R_{m+1}(s,N).
	\end{align*}
\end{proof}
We now provide the proof of Lemma \ref{claim}. 

\begin{proof}[Proof of Lemma \ref{claim}] Since the $m$\,--\,tuple $(x_{i_1},\ldots,x_{i_{m}})$ contributes to
	$R_{m}(\frac{s}{3},N)$, we have for all $j,k=1,2,\ldots, m$ that
	\begin{equation} \label{maxdistance}
	\lVert x_{i_j}- x_{i_k}\rVert \ls \lVert x_{i_j}- x_{i_1}\rVert + \lVert x_{i_1}- x_{i_k}\rVert \ls \frac{2s}{3N} \cdot \end{equation}
	For each $0\ls \ell \ls K-1$ the set $A_{\ell}$ can be written as a disjoint union 
	\begin{equation} \label{disjoint} 
	A_\ell = (A_\ell \cap A_{\ell-1}') \sqcup  (A_{\ell}' \cap A_{\ell-1}'') \sqcup  (A_\ell \cap A_{\ell}'')   \end{equation}
	where the three disjoint sets $A_\ell \cap A_{\ell-1}', A_{\ell}' \cap A_{\ell-1}''$ and $A_\ell \cap A_{\ell}''$ appearing are consecutive disjoint intervals of length $s/ 3N.$ \par We may assume without loss of generality that the points $x_{i_2}, x_{i_3},\ldots,x_{i_{m}}$ are in increasing order; that is, the signed distances of differences of consecutive terms are $(\!(x_{i_{n+1}}-x_{i_n})\!)\gs 0$ for $n=2,\ldots, m-1.$ \par We consider two different cases regarding the relative position of $x_{i_1}$ with respect to $x_{i_2}.$ If $(\!(x_{i_2}-x_{i_1})\!)\gs 0,$ let $0\ls \ell \ls K-1$ be such that $x_{i_1}\in A_{\ell}.$  Then $x_{i_1}$ lies in one of the three sets in the disjoint union  in \eqref{disjoint}; assume without loss of generality $x_{i_1}\in A_\ell \cap A_{\ell-1}'.$ Since $\|x_{i_1}-x_{i_m}\| \ls 2s/(3N)$ by \eqref{maxdistance} and the points $x_{i_1}, x_{i_2},\ldots, x_{i_m}$ are in increasing order, they will all lie in $A_{\ell}'\cap A_{\ell-1}''$ or $A_\ell\cap A_{\ell}''$, and therefore they will all lie in $A_\ell.$ Similarly, we see that if $x_{i_1}\in A_{\ell}' \cap A_{\ell-1}''$ then all points lie in $A_\ell '$ while if $x_{i_1}\in A_\ell \cap A_{\ell}''$ then they will all lie inside $A_{\ell}''.$ If $(\!(x_{i_1}-x_{i_2})\!)>0$, we repeat the same argument with the point $x_{i_2}$ in the place of $x_{i_1}.$  The lemma is now proved.

\end{proof}

\begin{rmk} Proposition \ref{HaukePrinciple} provides an inequality involving the correlation functions $R_k(\tfrac{s}{3},N)$ and $R_{k+1}(s,N)$ only for values of the scale $s>0$ that are large enough. This is not a restriction that comes from the method of proof followed, but rather a genuine obstruction, as can be seen by the following example. Define the sequence $(x_n)_{n\in\bN}$ by 
\[ x_n = \frac{1}{2^m}\big(2\big\lceil \tfrac{k}{2}\big\rceil -1\big) \quad \text{whenever } n=2^m+k,\, m\gs 0 \text{ and } 1\ls k \ls 2^m.  \]
That is, $x_1=x_2=0,$ $x_3=x_4=\frac12,$ $x_5=x_6=\frac14$, $x_7=x_8=\frac34$ etc. Then for the choice of the scale $s<2$ and for every integer of the form $N=2^m$ we have 
\[ R_2(s, N) = 1 \quad \text{ but } \quad R_3(s,N) = 0. \]
Using similar arguments we can define sequences for which $R_m(s,N)=1$ and $R_{m+1}(s,N)=0$ for some given $s>0$ and for infinitely many $N\gs 1.$  The upshot is that we cannot obtain an analogue of \eqref{conclusion} that holds for all values of $s>0.$
\end{rmk}

The next proposition connects the size of $R_k(s_1,\ldots,s_{k-1},N)$ with the size of $R_k^*(s_1,\ldots,s_{k-1},N)$ when the scales $s_1,\ldots, s_{k-1}$ are written in decreasing order.  

\begin{proposition}\label{star_lower_orders}
Let $s_1 \gs s_2 \gs \ldots \gs s_{k-1} > 0$. Defining $R_1(s,N) = 1$ for all $s>0$ and $N\gs 1$, we have
    \begin{equation*}
    R_k^*(s_1,\ldots,s_{k-1},N) \ls R_k(s_1,\ldots,s_{k-1},N) + \sum_{m= 1}^{k-1}b_mR_m(s_1,\ldots,s_{m-1},N), \end{equation*}
where $b_1,\ldots, b_{k-1} \in \mathbb{N}$ are constants depending only on $k$.
\end{proposition}
\begin{proof}
Observe that if a $k$\,--\,tuple $(i_1,i_2,\ldots,i_k)$ consists of distinct indices $i_1,\ldots, i_k\ls N$ then its contribution to $R_k^*(s_1,\ldots,s_{k-1},N)$ defined in \eqref{rkstardef} is the same as its contribution to $R_k(s_1,\ldots,s_{k-1},N),$ however the situation is different when the indices $i_1,\ldots, i_k$ are not pairwise distinct. \par 
Any $k$\,--\,tuple $\bfi=(i_1,i_2,\ldots,i_k)$ gives rise to a uniquely determined partition $\cP_{\bfi}=\{J_1, J_2, \ldots, J_m\}$ of $[k]=\{1,2,\ldots, k\}$ (where $m\ls k$) such that the following properties hold:
\begin{enumerate}
    \item[(i)]$i_j =i_\ell  \Leftrightarrow (j,\ell \in J_t\, \text{ for some } t\ls m)$
    \item[(ii)]$\min J_i < \min J_{i+1}, \quad i=1,\ldots, m-1.$
\end{enumerate}
Conversely, given a partition $\cP=\{J_1,J_2,\ldots, J_m\}$ of $[k]$ with the property that $\min J_t < \min J_{t+1}$ for $t=1,2,\ldots,m-1$, we define the correlation counting function $R_k^{\cP}=  R_k^{\cP}(s_1,\ldots,s_{k-1},N)$ by
\[ R_k^{\cP}=\# \frac{1}{N}\left\{  \bfi \in [N]^k, \cP_{\bfi}=\cP : \lVert x_{i_1} - x_{i_2}\rVert\ls \frac{s_1}{N},\ldots ,\|x_{i_{1}} - x_{i_k}\|\ls \frac{s_{k-1}}{N}\right\}.  \]
That is, $R_k^{\cP}$ is the variant of $R_k$ that counts correlations only over indices $i_1,\ldots,i_k\ls N$ with partition $\cP_{\bfi}$ equal to $\cP.$ \par
Now if for a fixed partition $\cP$ as above we set 
$$ j_r= \min J_r , \qquad r=1,2,\ldots, m$$ 
then we have that $j_r\gs r$. Write $J_r = \{ t_1 < t_2 <\ldots < t_s\}$ for one of the sets comprising $\cP.$ In view of the hypothesis that $s_1 \gs s_2 \gs \ldots \gs s_{k-1}$, the inequalities
\[\|x_{i_1} - x_{i_{t_\ell}}\| \ls \frac{s_{t_\ell-1}}{N} , \qquad \ell=1,\ldots,s \]
appearing in the definition of $R_k^{\cP},$ altogether imply that \[ \|x_{i_1} - x_{i_{t_\ell}}\| \ls  \frac{s_{t_\ell-1}}{N} \ls  \frac{s_{t_1-1}}{N} =  \frac{s_{j_r-1}}{N} \ls \frac{s_{r-1}}{N}, \quad \ell=1,\ldots, s. \]
Thus for the fixed partition $\cP=\{J_1,\ldots, J_m\}$ as above, we have  
\begin{align*}
R_k^{\mathcal{P}}(s_1,\ldots,s_{k-1},N) &\ls R_m(s_1,\ldots,s_{m-1},N).
\end{align*}
Finally, for the counting function $R_k^*$, summing over all possible partitions $\cP$ of $[k]$ we deduce that
\begin{align*}
R_k^*(s_1,\ldots,s_{k-1},N) &= \sum_{\cP}R_k^{\cP}(s_1,\ldots,s_{k-1},N) \\
&= \sum_{\cP: |\cP|=k}R_k^{\cP}(s_1,\ldots,s_{k-1},N) +\hspace{-2mm} \sum_{\substack{1\ls m \ls k-1\\ \cP : |\cP|=m}}\hspace{-2mm} R_k^{\cP}(s_1,\ldots,s_{k-1},N) \\[-2ex]
&\ls \,  R_k(s_1,\ldots,s_{k-1},N) + \sum_{m=1}^{k-1}b_m R_m(s_1,\ldots,s_{m-1},N),
\end{align*}
with $b_m \in \mathbb{N},\, m=1,\ldots, k-1$.
\end{proof}

\begin{rmk} (i)
 Proposition \ref{star_lower_orders} uses a combinatorial argument to derive a relation between the pair correlation function $R_k$ and $R_k^*,$ which is the corresponding sum over not necessarily distinct indices. A similar argument can  be found in \cite[Chapter 4]{rs2}. \par (ii)
The constants $c_i = c_i(m)$ appearing in the proof of Proposition \ref{HaukePrinciple} are the unsigned Stirling numbers of the first kind and the numbers $b_i=b_i(k)$ from Proposition \ref{star_lower_orders} are the Stirling numbers of the second kind. This can be easily seen from  \eqref{eleven} and the definition of the $b_i$ as the number of partitions of $[k]$ into $i$ nonempty subsets, respectively (see \cite{combinatorics} for more details).
\end{rmk}
 
\begin{corollary}\label{R_k_equiv_R_k*}
Let $\xn \subseteq [0,1]$ be an arbitrary sequence. For all $s>0$ large enough we have 
 \begin{equation}\label{R_k*_in_terms_of_R_k}
     R_k(s,N) \ls R_k^{*}(s,N) \ls R_k(s,N) + \mathcal{O}_k\left(\frac{1}{s}R_k(3^ks,N)\right), \qquad N\to \infty
 \end{equation}
 If $\xn$ has Poissonian k\,--\,th correlations, we have for $s$ sufficiently large
 \begin{equation} \label{34}
    R_k^*(s,N) = (2s)^{k-1} + \mathcal{O}(s^{k-2}), \qquad N\to \infty. 
\end{equation}
\end{corollary}
\begin{proof}
\eqref{34} follows immediately from \eqref{R_k*_in_terms_of_R_k} under the assumption of Poissonian k\,--\,th correlations.
Also, the first inequality of \eqref{R_k*_in_terms_of_R_k} is obvious. For the second inequality of \eqref{R_k*_in_terms_of_R_k},
 we use Proposition \ref{HaukePrinciple} and monotonicity of $R_k(s,N)$ in $s$ to deduce that for $1 \ls i \ls k-1,$   
 \[ R_i(s,N) \ll_{i,k} \frac{1}{s^{k-i}} R_k(3^{k-i}s,N) \ls \frac{1}{s} R_k(3^ks,N), \qquad N\to \infty \]
 for all $s>0$  sufficiently large. Using Proposition \ref{star_lower_orders}, the result follows.
 \end{proof}

Corollary \ref{R_k_equiv_R_k*} shows us that for $s$ large enough, we can work with $R_k^{*}$ instead of $R_k$. The function $R_k^{*}$ satisfies the following  inequalities that will be used in the proofs of Theorems \ref{main} and \ref{main2}.
\begin{proposition}\label{R_k^*_hölder_prop}
 {\it (i)} For any $s>0$ and $N\gs 1,$  
 \begin{equation} \label{R2andRk}
       R_2^*(s,N)^{k-1} \ls R_k^*(s,N).
 \end{equation}
{\it (ii)} Let $s_1,s_2,\ldots,s_{k-1} > 0$ and $N\gs 1$. Then  
    \begin{equation}\label{R_k^*_hölder}R_k^*(s_1,\ldots,s_{k-1},N)^{k-1} \ls  R_k^{*}(s_1,N) \, R_k^{*}(s_2,N)\cdots R_k^{*}(s_{k-1},N) .\end{equation}
\end{proposition}
\begin{proof}
For any $s>0$ and $N\gs 1$ we define
\[z_i(s,N) = \#\Big\{ j\ls N :  \lVert x_i - x_j\rVert \ls \frac{s}{N}  \Big\}, \quad  i \ls N.\] 
Under this notation, we observe that for $k\gs2$ and $s_1,s_2,\ldots,s_{k-1} > 0$ we have
\begin{equation}\label{R_k_easy_form}R_k^*(s_1,\ldots,s_{k-1},N)  = \frac{1}{N}\sum_{i\ls N} z_i(s_1,N)\cdot  \ldots \cdot  z_i(s_{k-1},N).\end{equation}
 For \eqref{R2andRk}, an application of the H\"older inequality with $p=k-1$ and $q=(k-1)/(k-2)$ gives 
\begin{equation*} 
    R_2^*(s,N)^{k-1} = \Big(\frac{1}{N}\sum_{i\ls N}z_i(s,N)\Big)^{k-1} \ls \frac{1}{N}\sum_{i\ls N}z_i(s,N)^{k-1}  = R_k^*(s,N).
\end{equation*}
Applying the Hölder inequality with exponents $p_i =  k-1, ( 1\ls i \ls k-1)$ to \eqref{R_k_easy_form} we obtain \eqref{R_k^*_hölder}.
\end{proof}

\section{Proof of Theorem \ref{main}}

Here we prove that sequences with Poissonian correlations of $k$\,--\,th order are uniformly distributed. 
We argue as in the proof of \cite[Theorem 2]{alp}: if the sequence $\xn$ is not uniformly distributed, then there exists some $a \in (0,1)$  such that relation $\lim\limits_{N \to \infty} \frac{1}{N}\#\{n \ls N: x_i \in [0,a]\}= a$ fails, and by the Bolzano\,--\,Weierstrass theorem there exists a  sequence $(N_j)_{j=1}^{\infty}$ and a number $b\neq a$ such that
 \begin{equation*}
 \lim\limits_{j \to \infty} \frac{1}{N_j}\#\{n \ls N_j : x_i \in [0,a]\} = b .
 \end{equation*}
 We need the following lemma, which tells us that under the assumption of Poissonian correlations of $k$\,--\,th order, the proportion of points in a ball with sufficiently fast shrinking radius has to be asymptotically zero. 
 
\begin{lem} \label{local} Let $k \gs 2.$ If $\xn$ is a sequence with Poissonian $k$\,--\,th correlations and $t\in [0,1],$ then for any $s>0$ we have 
\[ \lim_{N\to \infty} \frac{1}{N}\#\Big\{n\ls N : \|x_n-t\| \ls\frac{s}{2N}   \Big\} = 0.  \]
\end{lem} 
\begin{proof}
Assume for contradicition that there exist   $t\in [0,1], s>0$ and $\eta>0$ such that 
\[ \frac{1}{N}\#\Big\{n\ls N : \|x_n-t\| \ls\frac{s}{2N} \Big\} \gs \eta   \quad \text{ for inf. many } N\gs 1. \]
For such values of $N\gs 1$ we have 
\begin{align*}
    R_k(s,N) & \gs \frac{1}{N}\eta N \cdot (\eta N -1) \cdot \ldots \cdot (\eta N - (k-1)) \\
    &= \eta^k N^{k-1} + \mathcal{O}(N^{k-2}), \qquad N\to \infty,
\end{align*}
which contradicts the assumption of Poissonian $k$\,--\,th correlations.
\end{proof}
We can now apply Lemma \ref{local} with $t=a$ and $t=0$ to deduce that for $\varepsilon>0$ small enough, we  have
\begin{align*}
\frac{1}{N_j} \#\Big\{n\ls N_j:  \frac{s}{2N_j}  \ls x_n \ls a- \frac{s}{2N_j} \Big\} & =  \frac{1}{N_j}\#\{n \ls N_j : x_i \in [0,a]\} \\
& \qquad - \frac{1}{N_j} \#\Big\{n\ls N_j:   0 \ls  x_n < \frac{s}{2N_j}  \Big\} \\
& \qquad - \frac{1}{N_j} \#\Big\{n\ls N_j:   a- \frac{s}{2N_j} < x_n \ls a \Big\} \\
& \gs b-\varepsilon
\end{align*}
and also
\begin{align*}
\frac{1}{N_j} \#\Big\{n\ls N_j: a + \frac{s}{2N_j} \ls x_n \ls 1- \frac{s}{2N_j} \Big\} & \gs 1- b -\varepsilon
\end{align*}
for all $j\gs 1$ sufficiently large.\par  For these values of $j,$ if $F(t,s,N)$ is the function defined in \eqref{Fdef}, we see that
\begin{align*} \int_0^a F(t,s,N_j)\,\mathrm{d}t & = \int_0^a \sum_{n \ls N}\one_{B(x_n, \frac{s}{2N})}(t)\, \mathrm{d}t = \sum_{n\ls N} \lambda\Big(B(x_n,\frac{s}{2N}) \cap [0,a] \Big) \\
&\gs \frac{s}{N_j} \#\Big\{n\ls N_j:  \frac{s}{2N_j} \ls x_n \ls a- \frac{s}{2N_j} \Big\} \\
& \gs s(b-\varepsilon)  
\end{align*}
 and similarly
\[  \int_a^1 F(t,s,N_j)\mathrm{d}t \gs s(1 - b - \varepsilon). \]
 Applying the Cauchy\,--\,Schwarz inequality we deduce that
\begin{eqnarray*}
 \int_0^1 F(t,s,N_j)^2 \,\mathrm{d}t &  = &  \int_0^a F(t,s,N_j)^2\,\mathrm{d}t   +    \int_a^1 F(t,s,N_j)^2 \,\mathrm{d}t  \\
	&\gs&  \frac{1}{a}\left( \int_0^a F(t,s,N_j)\,\mathrm{d}t\right)^2 + \frac{1}{1-a}\left( \int_a^1 F(t,s,N_j)\,\mathrm{d}t\right)^2 \\
    &\gs & \frac{s^2}{a}(b-\varepsilon)^2 + \frac{s^2}{1-a}(1-b-\varepsilon)^2.
\end{eqnarray*}
Since $a\neq b,$ if we choose $\varepsilon>0$ small enough we have 
\[   \frac{(b-\varepsilon)^2}{a} + \frac{(1-b-\varepsilon)^2}{1-a} = 1 + \delta \quad \text{ for some } \delta > 0 \]
and therefore 
\begin{align} \label{big2ndmoment}
    \int_0^1 F(t,s,N_j)^2 \,\mathrm{d}t \gs (1+\delta)s^2
\end{align}
for all $j\gs 1$ large enough. 
At this point, we present a simple fact that connects the correlation functions $R_2^*(s,N)$ with the function $F(t,s,N)$.% defined in \eqref{Fdef}. 

\begin{lem} \label{lemma11} 
For any $s>0$, we have 
\begin{equation} \label{FandR}
    \int_0^1 F(t,s,N)^2\,\mathrm{d}t = \int_0^s R_2^*(\sigma, N)\,\mathrm{d}\sigma . 
 \end{equation}
\end{lem}
\begin{proof}
By definition of $F(t,s,N)$, we see that
\begin{equation*}\int_0^1 F(t,s,N)^2\,\mathrm{d}t = \sum_{m,n\ls N} \lambda\Big( B(x_m,\frac{s}{2N})\cap B(x_n,\frac{s}{2N}) \Big) = \sum_{m,n\ls N} \Big\{ \frac{s}{N} - \|x_m - x_n\| \Big\}^+, \end{equation*}
while the definition of $R_2^*(s,N)$ gives 
\begin{equation*} \int_0^s R_2^*(\sigma,N)\,\mathrm{d}\sigma = \int_0^s \frac{1}{N}\sum_{m,n\ls N} \one_{[N\|x_m-x_n\|, \infty)}(\sigma)\,\mathrm{d}\sigma = \frac{1}{N}\sum_{m,n\ls N} \Big\{ s - N \|x_m - x_n\| \Big\}^+  \end{equation*}
and these two terms on the right\,--\,hand sides of the equations above are clearly equal. \end{proof}

Combining \eqref{big2ndmoment} with \eqref{FandR} we get 
\begin{equation} \label{lowerbound}
    \int_0^s R_2^*(\sigma,N_j)\,\mathrm{d}\sigma \gs (1 + \delta)s^2, 
\end{equation}
and in turn using  \eqref{lowerbound} with \eqref{R2andRk} and  \eqref{34} we see that 
\begin{align*}
(1+\delta)s^2  & \ls  \int_0^s R_2^*(\sigma,N_j) \,\mathrm{d}\sigma  \ls \int_0^s R_k^*(\sigma,N_j)^{\frac{1}{k-1}} \,\mathrm{d}\sigma = s^2 + \mathcal{O}(s), \quad j\to \infty 
\end{align*}
which is a contradiction for values of $s$ which are sufficiently large. 
 
 \section{ Proof of Theorem  \ref{main2}}
  In this section we present the proof of Theorem \ref{main2}, which generalises Theorem B in the context of $k$\,--\,th order correlations. We first present the properties of the correlation functions $C_k(s,N)$ and $C_k^*(s,N)$ that we use in the proof and then continue with the proof itself.
 
 \subsection{The functions $C_k, C_k^*$.} 
For convenience, when $N\gs 1$ and $1\ls i, j \ls N$ we write 
\begin{equation} \label{lambda_definition} 
\lambda_N(s;i,j) = \lambda\Big(B\big(x_i, \frac{s}{2N} \big)\cap B\big(x_j, \frac{s}{2N} \big) \Big). 
\end{equation}
As shown in the following proposition, the values of $C_k$ and $C_k^*$ can be expressed explicitly in terms of the numbers $\lambda_N(s;i,j).$

\begin{proposition}\label{ck_prop}
The functions $C_k, C_k^*$ satisfy
$$ C_k(s_1,\ldots,s_{k-1},N) = N^{k-2}\hspace{-2mm}\sum_{\substack{i_1,\ldots,i_k\ls N \\ \mathrm{distinct}}}\hspace{-2mm}\lambda_{N}(s_1;i_1,i_2)\lambda_{N}(s_2;i_1,i_3)\ldots \lambda_{N}(s_{k-1};i_1,i_k)\vspace{-3mm} $$
and
$$C_k^{*}(s_1,\ldots,s_{k-1},N) = N^{k-2}\hspace{-2mm}\sum_{i_1,\ldots,i_k\ls N}\hspace{-2mm}\lambda_{N}(s_1;i_1,i_2)\lambda_{N}(s_2;i_1,i_3)\ldots \lambda_{N}(s_{k-1};i_1,i_k). $$
\end{proposition}

\begin{proof}
    Note that for any $i, j \ls N$ and $s>0$,
    \[\int_{0}^{s}\!\mathds{1}_{\left[0,\tfrac{\sigma}{N}\right]}\!\left(\lVert x_i - x_j \rVert\right)\mathrm{d}\sigma =
    \left\{s - N\lVert x_i - x_j \rVert\right\}^{+} =
    N\lambda_N(s;i,j).
    \]
Therefore,
\begin{align*}
C_k^{*}(s_1,\ldots,s_{k-1},N) &= \int_{0}^{s_1}\!\!\int_{0}^{s_2}\!\cdots\!\int_{0}^{s_{k-1}} R_k^{*}(\sigma_1,\ldots,\sigma_{k-1},N)\, \mathrm{d}\sigma_1
    \mathrm{d}\sigma_2\ldots \mathrm{d}\sigma_{k-1}
    \\&= \frac{1}{N}\sum_{i_1,\ldots,i_{k}\ls N}\prod_{j=1}^{k-1} \int_{0}^{s_j} \mathds{1}_{\left[0,\tfrac{\sigma}{N}\right]}(\lVert x_1 - x_{j+1}\rVert)\, \mathrm{d}\sigma 
    \\&= \frac{1}{N}\sum_{i_1,\ldots,i_{k}\ls N} \prod_{j=1}^{k-1}N\lambda_N(s_j;i_1,i_{j+1})
    \\&= N^{k-2}\sum_{i_1,\ldots,i_k\ls N} \lambda_{N}(s_1;i_1,i_2)\lambda_{N}(s_2;i_1,i_3)\ldots\lambda_{N}(s_{k-1};i_1,i_k).
    \end{align*}
The proof follows similarly for $C_k(s_1,\ldots,s_{k-1},N)$.     
\end{proof}
The next proposition gives a lower bound for the size of $C_k^{*}$. For convenience, when we deal with $\lambda_N(s;i,j)$ as in \eqref{lambda_definition} and the value of $s>0$ is clear from the context, we  suppress the dependence on $s$ and simply write $\lambda_N(i,j).$
\begin{proposition}\label{characterization_ck} Let $(x_n)_{n \in \mathbb{N}}\subseteq [0,1]$ be a sequence. For any $s>0$, 
\begin{equation*}%\label{minimalck}  
C_k^{*}(s,N) \gs s^{2(k-1)} \qquad \text{ for all }\, N\gs 1.
\end{equation*}

\end{proposition}
\begin{proof}
By Proposition \ref{ck_prop}, we have
\begin{align*}
    C_k^{*}(s,N) = N^{k-2}\sum_{i_1,\ldots,i_k\ls N}\lambda_N(i_1,i_2)\lambda_N(i_1,i_3)\ldots \lambda_N(i_1,i_k).
\end{align*}
Using the Hölder inequality with $p=(k-1)/(k-2)$ and $q=k-1$, we get
\begin{align*}
    C_2^{*}(s,N)^{k-1} &= \Bigg(\sum_{i_1,i_2\ls N} \lambda_N(i_1,i_2)\Bigg)^{k-1}\\ 
    &\ls N^{k-2}\sum_{i_1 = 1}^{N}\Bigg(\sum_{i_2 = 1}^{N}\lambda_N(i_1,i_2)\Bigg)^{k-1}  = C_k^{*}(s,N).
\end{align*}
On the other hand, if $F(t,s,N)$ is the function defined in \eqref{Fdef}, by Lemma \ref{lemma11} we have \[C_2^{*}(s,N) = \int_0^s R_2^*(\sigma,N)\mathrm{d}\sigma = \int_{0}^{1} F(t,s,N)^2 \mathrm{d}t \gs \left(\int_{0}^{1} F(t,s,N) \mathrm{d}t\right)^2 = s^2,
\]
which implies that for any $N\gs 1$ we have
$C_k^{*}(s,N) \gs s^{2(k-1)} .$ 
\end{proof}
For the  proof of Theorem \ref{main2}, we need to introduce a localised version of the correlation functions $C_k^*(s,N)$.  For any $s>0$ and any interval $A \subseteq [0,1)$ we define the quantity 
\begin{equation} \label{ckrestricted}
C_k^{*}(A;s,N) := N^{k-2}\mathop{\sum_{ i_1,\ldots,i_k\ls N}}_{x_{i_1},\ldots, x_{i_k} \in A}\hspace{-1mm} \lambda_N(i_1,i_2)\lambda_N(i_1,i_3)\ldots\lambda_N(i_1,i_k). \end{equation}
In view of Proposition \ref{ck_prop}, one can think of $C_k^*(A;s,N)$ as a restriction of $C_k^*(s,N)$ on $A$. 
We intend to use the obvious fact that for any partition $(A_j)_{j=1}^{M}$ of the unit interval we have
\begin{equation}\label{C_k_split}
C_k^{*}(s,N) \gs \sum_{j = 1}^M C_k^{*}(A_j;s,N).\end{equation}
The following proposition generalises Proposition \ref{characterization_ck} for the localised versions of the $C_k^*(s,N)$ in the context of Theorem \ref{main2}. 

\begin{proposition}\label{k_th_moment_step}
 Let $G:[0,1]\to\bR $ be an  asymptotic distribution function of $\xn$   and $(N_j)_{j\in\bN}$ be a sequence as in \eqref{G_definition}. Let $k \gs 2$, $A \subseteq [0,1]$ be an interval,  $a := \lambda(A)$ be its Lebesgue measure and $b :=\mu_G(A)$ be its Riemann\,--\,Stieltjes measure with respect to $G$. Then for all $s>0$ we have
 \begin{equation*}
     \liminf_{j \to \infty} C_k^{*}(A,s,N_j) \gs \frac{b^k}{a^{k-1}}s^{2(k-1)}.
 \end{equation*}
\end{proposition}
\begin{proof}
Let $s >0$ arbitrary and assume without loss of generality that $A = [0,a]$.
Let $(x_{r_n})_{n=1}^{\infty}$ denote the subsequence of  $(x_n)_{n=1}^{\infty}$ consisting of all terms that lie in $[0,a]$ and define the sequence $(z_n)_{n\in\bN}$ by 
 $$z_n  = \frac{1}{a}x_{r_n}, \qquad n\gs 1.$$
 We wish to establish a relation between $C_k^*([0,a],s,N)$ defined in \eqref{ckrestricted} and the correlation counting function $C_k^*(s,N)$ relevant to the sequence $(z_n)_{n\in\bN}.$ Since we need to specify to which sequence the  counting function refers to, from now on we write  $C_k^*( (z_n)_{n\in\bN}, s, N)$ for the function $C_k^*$ that refers to $(z_n)_{n\in\bN}$, and when we do not state which sequence the correlation function $C_k^*$ refers to, it will be understood that it refers to $(x_n)_{n\in\bN}.$ \par 

  We define
 \[K_{N} = \#\{i \ls N: x_i \in [0,a]\}, \qquad N\gs 1. \]
 Let $\varepsilon >0$. By the definitions of $a$ and $b$, there exists  $J\gs 1$ such that for all $j \gs J,$
 \begin{equation}\label{eps_close}
 (b-\varepsilon)N_j \,\, < \,\, \#\{i \ls N_j: x_i \in [0,a]\} \,\, <\,\, (b+\varepsilon)N_j. 
 \end{equation}
 Then for any $s>0$ we have
 \begin{align}
C_{k}^{*}\left([0,a],s,N_j\right) &=
N_j^{k-2}\mathop{\sum_{i_1,\ldots,i_k\ls N_j}}_{ x_{i_1},\ldots,x_{i_k} \ls a} \prod_{2\ls p \ls k} \lambda\Big(B\Big(x_{i_1},\frac{s}{2N_j}\Big)\cap B\Big(x_{i_p},\frac{s}{2N_j}\Big)\Big)\nonumber \\
 &= N_j^{k-2} \hspace{-2mm}  \sum_{ i_1,\ldots, i_k \ls K_{N_j}}  \prod_{2\ls p \ls k}  \lambda\Big(B\Big(az_{i_1},\frac{ s}{2 N_j}\Big)\cap B\Big(az_{i_p},\frac{  s}{2 N_j}\Big)\Big). \label{C_k_0,a} \end{align}
At this point, we notice that when  $N_j$ is sufficiently large, the measure of the intersections appearing in the right-hand side of \eqref{C_k_0,a} is
 \begin{align*} 
 \lambda\Big(B\Big(az_{i_1},\frac{s}{2N_j}\Big)\cap B\Big(az_{i_p},\frac{ s}{2 N_j}\Big)\Big)   = a\lambda\Big(B\Big(z_{i_1},\frac{ s}{2a N_j}\Big)\cap B\Big(z_{i_p},\frac{s}{2a N_j}\Big)\Big)& \\
 \stackrel{\eqref{eps_close}}\gs a\lambda\Big(B\Big(z_{i_1},\frac{s(b-\varepsilon)}{2aK_{N_j}}\Big)\cap B\Big(z_{i_p},\frac{s(b-\varepsilon)}{2aK_{N_j}}\Big)\Big).&
 \end{align*}
 Inserting this into \eqref{C_k_0,a} gives
 \begin{align*}
  C_{k}^{*}\left([0,a],s,N_j\right) &\gs N_j^{k-2} \hspace{-4mm} \sum_{ i_1,\ldots, i_k \ls K_{N_j}} \hspace{-4mm} a^{k-1}\hspace{-2mm} \prod_{2\ls p \ls k}\lambda\Big(B\Big(z_{i_1},\frac{s(b-\varepsilon)}{2aK_{N_j}}\Big)\cap B\Big(z_{i_p},\frac{s(b-\varepsilon)}{2aK_{N_j}}\Big)\Big) \\
 &=  a^{k-1} \frac{N_j^{k-2}}{K_{N_j}^{k-2}} \cdot C_k^*\big((z_n)_{n\in\bN}; s(b-\varepsilon) a^{-1} , K_{N_j}\big) \\
 &\stackrel{\eqref{eps_close}}\gs   \frac{a^{k-1}}{(b+\varepsilon)^{k-2}}  C_k^*\big((z_n)_{n\in\bN}; s(b-\varepsilon)a^{-1} , K_{N_j}\big).
 \end{align*}
 We now use Proposition \ref{characterization_ck} for $C_k^*\big((z_n)_{n\in\bN}; (b-\varepsilon)s a^{-1} , K_{N_j}\big) $ to deduce that
 \[ C_{k}^{*}\left([0,a],s,N_j\right) \gs \frac{a^{k-1}}{(b+\varepsilon)^{k-2}} \cdot \frac{(b-\varepsilon)^{2(k-1)}}{a^{2(k-1)}} \, s^{2(k-1)}\]
for all  $j\gs J$.  Since $\varepsilon>0$ can be chosen arbitrarily small, this finally implies 
 \[ 
 \liminf_{j \to \infty} C_{k}^{*}\left([0,a],s,N_j\right) \gs
 \frac{b^k}{a^{k-1}}s^{2(k-1)}.\]
\end{proof}
 
\subsection{Proof of Theorem \ref{main2}}  For each $r\gs 1$ we define the intervals 
\begin{equation}\label{intervals_B} B_{r,i} = \Big[ \frac{i}{2^r}, \frac{i+1}{2^r} \Big), \qquad 0 \ls i \ls 2^r -1.\end{equation}
Applying Proposition \ref{k_th_moment_step} and \eqref{C_k_split} to the partition $(B_{r,i})_{i=0}^{2^r-1},$ we deduce that for any $s > 0$
 \begin{align}\label{basis} 
 \frac{\liminf_{j\to\infty} C_k^*(s,N_j)}{s^{2(k-1)}} &\gs  \sum_{i=0}^{2^r-1} 2^{r(k-1)}
 \left(G\Big(\frac{i+1}{2^r}\Big) - G\Big(\frac{i}{2^r}\Big)\right)^k.
 \end{align}
 We now  consider two cases regarding the function $G.$ \newline 
$\text{(i)}$ If $G$ is not absolutely continuous,  there exists a fixed $\varepsilon> 0$  such that for any $\delta > 0$, there exist $M = M_{\delta} \in \mathbb{N}$ many pairwise disjoint intervals
    $I_1,\ldots, I_{M} \subseteq (0,1)$ with $I_j = (a_j,b_j), j = 1,\ldots,M$
    such that 
    \begin{equation*}
    \sum_{j = 1}^{M}(b_j -a_j) < \frac{\delta}{2},\qquad 
    \sum_{j = 1}^{M} (G(b_j) -G(a_j)) > \varepsilon.
    \end{equation*}
    For $r \in \mathbb{N}$, we define 
    \[a_{j,r} = \frac{\lfloor 2^r a_j\rfloor}{2^r}, \qquad b_{j,r} = \frac{\lceil 2^r b_j\rceil}{2^r}, \qquad
    I_{j,r} = (a_{j,r},b_{j,r}).\]
    In other words, $I_{j,r}$ is the smallest interval of the form $(\frac{m}{2^r}, \frac{n}{2^r})$ with $m,n \in \mathbb{N}$
    that contains $I_j$. We choose $r = r(\delta)\gs 1$  large enough so that the intervals $I_{j,r},\, j = 1,\ldots,M$ are still pairwise disjoint, contained in $(0,1)$ and also $2^r > 4M/\delta$. We then have
    \begin{equation}\label{still_smaller_del}
    \sum_{j = 1}^{M} (b_{j,r} -a_{j,r}) \ls
    \sum_{j = 1}^{M} \left(b_{j} -a_{j} + \frac{2}{2^r}\right)
    < \frac{\delta}{2} + \frac{\delta}{2} = \delta. \end{equation}    
  Let $J_{r,\delta} \subseteq \{0,\ldots,2^r -1\}$ be the index set such that
    \[\bigcup_{i \in J_{r,\delta}} B_{r,i} = \bigcup_{j = 1}^{M} I_{j,r}.\]
    Using \eqref{still_smaller_del}, we see that 
    \begin{equation}\label{I_small}\lvert J_{r,\delta} \rvert <  2^r \delta.\end{equation}
    Additionally, since $G$ is non-decreasing, we still have
    \begin{equation}\label{still_larger_eps}\sum_{i \in J_{r,\delta}} \left(G\Big(\frac{i+1}{2^r}\Big) - G\Big(\frac{i}{2^r}\Big)\right) > \varepsilon.\end{equation}
Combining \eqref{basis} with \eqref{still_larger_eps}, applying the H\"{o}lder inequality in the form $\left(\sum_{i \in J} x_i\right)^k \ls \lvert J \rvert^{k-1}\sum_{i \in J} x_i^k$ and using \eqref{I_small} we obtain
 \begin{align*}
 \frac{\liminf_{j\to\infty} C_k^*(s,N_j)}{s^{2(k-1)}} &\gs   2^{r(k-1)} \sum_{i \in J_{r,\delta}}\!
 \left(\!G\Big(\frac{i+1}{2^r}\Big) - G\Big(\frac{i}{2^r}\Big)\!\right)^k \\
 &\gs \frac{2^{r(k-1)}}{\lvert J_{r,\delta}\rvert^{k-1}} \Bigg(\sum_{i \in J_{r,\delta}}\!\! \Big(\!G\Big(\frac{i+1}{2^r}\Big) - G\Big(\frac{i}{2^r}\Big)\!\Big)\Bigg)^k
 \gs \frac{\varepsilon^k}{\delta^{k-1}}\cdot 
 \end{align*}
   As $\varepsilon > 0$ is fixed and $\delta>0$ can be chosen arbitrarily small, we obtain that for any $s > 0,$  $\liminf_{j\to\infty} C_k^*(s,N_j) = \infty,$ or equivalently $\lim_{j\to\infty}C_k^*(s,N_j) = \infty.$ 
  By monotonicity of $R_k^{*},$ this immediately implies that also $ \lim_{j\to\infty} R_k^{*}(s, N_j) = \infty.$ \par We will now show that $\lim_{j\to\infty}R_k(s,N) = \infty. $ Letting $z_i=z_i(s,N)$ be as in the proof of Proposition \ref{R_k^*_hölder_prop}, for all $m<k$ and $N\gs 1$ we have
  \begin{align*}R_m(s,N) &= \frac{1}{N}\sum_{i \ls N} (z_i-1)(z_i-2)\cdot\ldots\cdot(z_i-(m-1))
  \\ &= \frac{1}{N}\mathop{\sum_{i \ls N, z_i \gs k}}\hspace{-2mm} (z_i-1) \cdot\ldots\cdot(z_i-(m-1))
  + \frac{1}{N} \mathop{\sum_{i \ls N, z_i < k}}\hspace{-2mm} (z_i-1) \cdot\ldots\cdot(z_i-(m-1))
  \\ &\ls  \frac{1}{N}\mathop{\sum_{i \ls N, z_i \gs k}}  (z_i-1)\cdot\ldots\cdot(z_i-(m-1))
  + k^m \\ &\ls R_k(s,N) + k^m.
  \end{align*}
  Together with Proposition \ref{star_lower_orders}, this gives
  \begin{equation}\label{new_ineq}R_k^{*}(s,N) \ls (1 + b_{k-1} + \ldots + b_1)R_k(s,N) + \sum_{m = 1}^{k-1}b_{m} k^m.\end{equation}
  Since $ \lim_{j\to\infty} R_k^{*}(s, N_j) = \infty,$ \eqref{new_ineq} immediately shows that also 
  \[ \lim_{j \to \infty} R_k(s,N_j) = \infty,\] as required. \newline
 $\text{(ii)}$ We now consider the case when $G$ is absolutely continuous, and write $g$ for the function as in the hypothesis of Theorem \ref{main2}. We follow the method used in \cite{alp} to prove Theorem B. We first assume that $g^k$ is integrable.
 For each $r\gs 1,$ let $\mathcal{F}_r$ be the $\sigma$\,--\,algebra generated by the intervals $B_{r,i},\; 0 \ls i \ls 2^r -1$ defined in \eqref{intervals_B}.
 By the hypothesis, for any $r\gs 1$ and $i=0,1,\ldots, 2^r-1$ we have 
 \[ \lim_{j\to\infty}\frac{1}{N_j}\#\{n\ls N_j : x_n \in B_{r,i} \} = G\Big(\frac{i+1}{2^r}\Big) - G\Big(\frac{i}{2^r}\Big) = \int_{B_{r,i}}g(x)\mathrm{d}x. \]
Inequality \eqref{basis} implies that for any $s > 0$
 \[ \frac{\liminf_{j\to\infty} C_k^*(s,N_j)}{s^{2(k-1)}} \gs  \sum_{i=0}^{2^r-1} 2^{r(k-1)}\left( \int_{B_{r,i}}\hspace{-2mm}g(x)\mathrm{d}x \right)^k = \int_0^1 \mathbb{E}\!\left[g\vert \mathcal{F}_r\right]\!(x)^k\mathrm{d}x. \]
 (Here the function $\mathbb{E}\left[g\vert \mathcal{F}_r\right]$ is the conditional expectation of $g$ with respect to the $\sigma$\,--\,algebra $\mathcal{F}_r,$ for the definition we refer to \cite[p. 121]{ergodic}.)  
At this point, we observe that $\mathcal{F}_r \subseteq \mathcal{F}_{r+1}$ for all $r\gs 1$ and the $\sigma$\,--\,algebra generated by $\bigcup\limits_{r=1}^\infty \mathcal{F}_r$ is the Borel $\sigma$\,--\,algebra on $[0,1].$ Since $g$ has a finite $k$\,--\,th moment, we can apply the martingale convergence theorem \cite[Thm 5.5]{ergodic} to deduce
\[\lim_{r \to \infty} \int_0^1 \mathbb{E}\!\left[g\vert \mathcal{F}_r\right]\!(x)^k\mathrm{d}x =  \int_{0}^{1} g(x)^k \, \mathrm{d}x,\]
so we obtain
\begin{equation}\label{C_k_k_thmoment}\frac{\liminf_{j \to \infty} C_k^{*}(s,N_j)}{s^{2(k-1)}} \gs \int_{0}^{1} g(x)^k \,\mathrm{d}x.\end{equation}
Applying Proposition \ref{R_k^*_hölder_prop}, we see that 
    \begin{align}
    \liminf_{j \to \infty} C_k^{*}(s,N_j) &\ls \limsup_{j \to \infty}\int_{[0,s]^{k-1}} R_k^{*}(t_1,\ldots,t_{k-1},N_j) \,\mathrm{d}t_1 \ldots \mathrm{d}t_{k-1} \nonumber
    \\&\ls \limsup_{j \to \infty}\int_{[0,s]^{k-1}} R_k^{*}(t_1,N_j)^{\frac{1}{k-1}} \cdots R_k^{*}(t_{k-1},N_j)^{\frac{1}{k-1}} \,\mathrm{d}t_1 \ldots \mathrm{d}t_{k-1} \label{applied_hölder_thm2}
    \\&= \limsup_{j \to \infty}\left(\int_{0}^s  R_k^{*}(t,N_j)^{1/(k-1)} \,\mathrm{d}t\right)^{k-1}. \nonumber
    \end{align}
We will first prove that
\begin{equation}\label{alternative_statement}\limsup_{s \to \infty} \frac{\limsup_{j \to \infty}R_k^{*}(s,N_j)}{(2s)^{k-1}}\gs  \int_{0}^{1} g(x)^k \,\mathrm{d}x\end{equation}
and prove the same with $R_k$ in place of $R_k^{*}$ later. Assume for contradiction that there exists an $\varepsilon > 0$ and some $S_0=S_0(\varepsilon) \in \mathbb{N}$ such that for all $s > S_0,$ 
\begin{equation}\label{contr_thm2}\limsup_{j \to \infty} R_k^{*}(s,N_j) < {(2s)^{k-1}}\left(\int_{0}^{1} g(x)^k \,\mathrm{d}x - \varepsilon\right).\end{equation}
Raising both sides in \eqref{contr_thm2} to the power of $1/(k-1)$ and integrating, we see that for all $s>S_0$ we have
\begin{align} 
    \int_0^s\!\!\limsup_{j\to\infty} R_k^*(\sigma,N_j)^{\frac{1}{k-1}}\mathrm{d}\sigma & =   \int_{S_0}^s \limsup_{j\to\infty} R_k^*(\sigma,N_j)^{\frac{1}{k-1}}\mathrm{d}\sigma +\mathcal{O}(1) \nonumber \\
& \ls s^2\left(\int_{0}^{1} g(x)^k \,\mathrm{d}x -\varepsilon\right)^{\frac{1}{k-1} } + \mathcal{O}(1). \label{int2}
\end{align}
The $\mathcal{O}(1)$ term here comes from the fact that $\limsup_{j\to\infty}R_k^*(\sigma,N_j)$ is bounded in the range $(0,S_0].$
Combining \eqref{C_k_k_thmoment}, \eqref{applied_hölder_thm2} and \eqref{int2} and applying the reverse Fatou  Lemma, we obtain
\begin{align*}s^2\left(\int_{0}^{1} g(x)^k \,\mathrm{d}x\right)^{\frac{1}{k-1}}
&\ls  \liminf_{j \to \infty} C_k^{*}(s,N_j)^{\frac{1}{k-1}} \\ & \ls \limsup_{j \to \infty} \int_{0}^s R_k^{*}(\sigma,N_j)^{1/(k-1)}\,\mathrm{d}\sigma
\\&\ls\int_{0}^s \limsup_{j \to \infty} R_k^{*}(\sigma,N_j)^{1/(k-1)}\,\mathrm{d}\sigma \\
&\ls s^2 \left(\int_{0}^{1} g(x)^k \,\mathrm{d}x - \varepsilon\right)^{\frac{1}{k-1}} + \mathcal{O}(1),
\end{align*}
a contradiction when $s \to \infty$. \par 
To prove \eqref{alternative_statement} with $R_k$ instead of $R_k^{*}$, we argue as follows: assume for contradiction that 
\begin{equation}\label{contr_thm2_error}\limsup_{s \to \infty} \frac{\limsup_{j \to \infty} R_k(s,N_j)}{(2s)^{k-1}}< \int_{0}^{1} g(x)^k \,\mathrm{d}x. \end{equation}
Combining \eqref{alternative_statement} with \eqref{contr_thm2_error},  we deduce that there exists some $\delta > 0$ such that
\begin{equation*}
\limsup_{s\to \infty} \frac{\limsup_{j\to \infty}(R_k^*(s,N_j) - R_k(s,N)) }{(2s)^{k-1}} \gs  \delta \int_{0}^{1} g(x)^k \,\mathrm{d}x.
\end{equation*}
In view of \eqref{R_k*_in_terms_of_R_k}, this implies that
\[\limsup_{s \to \infty} \frac{1}{s}\frac{\limsup_{j \to \infty} R_k(3^ks,N_j)}{(2s)^{k-1}} \gs   M \,\int_{0}^{1} g(x)^k \,\mathrm{d}x,\]
where the constant $M>0$ depends on $k$; 
\iffalse With $t = 3^ks$, we obtain
\[\limsup_{t \to \infty} \frac{1}{t}\frac{\lim_{j \to \infty} R_k(t,N_j)}{(2t)^{k-1}} \gg_k  \varepsilon\int_{0}^{1} g(x)^k \,\mathrm{d}x,\] \fi
a contradiction to \eqref{contr_thm2_error}.
 
If $g$ does not have a finite $k$\,--\,th moment, we approximate $g$ by truncations 
\[g_\ell(x) := \begin{cases} g(x), &\text{ if } g(x) \ls \ell\\
0,  &\text{ if } g(x) > \ell. \end{cases} \]
We can apply the arguments from above to show that for any $\ell \in \mathbb{N},$
\[\limsup_{s \to \infty} \frac{\limsup_{j \to \infty}R_k(s,N_j)}{(2s)^{k-1}}\gs  \int_{0}^{1} g_{\ell}(x)^k \,\mathrm{d}x.\]
Since $\int_{0}^1 g_{\ell}(x)^k \, \mathrm{d}x$ can be made arbitrarily large, the result follows.

%%%%%%%%%%%%%%HERE STARTS THE PROOF OF THE ALTERNATIVE THEOREMS AND COROLLARIES%%%%%%%%%%%%%%%%%%%%%%%%%%%%%%%%%
\section{Proof of Theorems \ref{thm3} and  \ref{ap_thm} }
The proof of Theorem \ref{thm3} follows straightforwardly from Proposition \ref{HaukePrinciple}. Assume $s_1,\ldots, s_{k-1}>0$ are such that 
\[ \limsup_{N\to\infty} R_k(s_1,\ldots,s_{k-1},N) = \infty. \]
If $p>k$, then by \eqref{conclusion} for any $s$ large enough with respect to $k$ and $s_1,\ldots, s_{k-1}$ we have $\limsup_{N\to\infty}R_p(s,N)=\infty.$ Therefore the sequence $\xn$ cannot have Poissonian correlations of order $p$. 

Turning to Theorem \ref{ap_thm}, statement (i) follows by Bourgain's construction in \cite[Appendix]{all} of a subset $\mathcal{A}=(a_n)_{n\in \bN}$ of the positive integers such that $ E(\mathcal{A}_N)= o(N^3), N\to \infty$ and for almost all $x\in[0,1]$ the sequence $(a_n x)_{n\in\bN}$ satisfies $$\limsup_{N\to\infty}R_2(1,N) = \infty.  $$
By Theorem \ref{thm3}, for almost all $x\in [0,1]$ the sequence $(a_n x)_{n=1}^{\infty}$ does not have Poissonian correlations of any order $k\gs 2.$ \par 
For statement (ii), within the context of Theorem \ref{ap_thm} we write $R_k(s,N,x)$ for the $k$\,--\,th correlation function $R_k(s,N)$ of the sequence $(a_n x)_{n\in\bN}.$  Also given any finite set $A$, we write 
\[T(A) := \# \{(a,b,c) \in A^3:  a-b = b-c \neq 0\}.\]
\par 
 We shall make use of a result in additive combinatorics, which states that whenever the additive energy of a set $A$ is $E(A) \gs \kappa_1 |A|^3$ then $T(A) \gs \kappa_2|A|^2$; the constant $\kappa_2>0$ only depends on $\kappa_1>0.$ 
(see e.g. \cite[Theorem 6.1]{pr}).

Therefore the assumption of Theorem \ref{ap_thm} implies that 
   \[ T(\mathcal{A}_N) \gs c N^2 \qquad \text{ for infinitely many } N\gs 1,  \]
where $c > 0$ is a constant. \par

    Let $(i,j,k) \in \mathcal{A}_N^3$ be a 3\,--\,term arithmetic progression with distance $d \gs 1$. Then we have
    $\lVert j\alpha - i\alpha \rVert \ls \frac{s}{N}$ and  $\lVert k\alpha - j\alpha \rVert \ls \frac{s}{N}$
    if and only if
    $\lVert d\alpha \rVert \ls \frac{s}{N}$, which happens for $\alpha$ on a set of measure $2s/N,$ independent of the value of $d\gs 1.$
    Observe that for infinitely many values of $N\gs 1$, we have
    \begin{align*}
        \int_{0}^{1} R_3(s,N,\alpha) \,\mathrm{d}\alpha & =  
        \sum_{\substack{i, j, k \in \mathcal{A}_N \\\text{distinct}}} \int_{0}^{1} \frac{1}{N} \mathds{1}_{[\lVert j\alpha - i\alpha \rVert \ls \frac{s}{N}, \lVert k\alpha - j\alpha \rVert \ls \frac{s}{N}]}(\alpha) \,\mathrm{d}\alpha\\
        &\gs \frac{1}{N}\hspace{-4mm}\sum_{\substack{(i,j,k) \text{ is a} \\ \text{non-trivial 3-AP}}}\hspace{-3mm}\frac{2s}{N}\,\, =\,\, 2s\frac{T(\mathcal{A}_N)}{N^2}\,\, \gs \,\, 2sc.
    \end{align*}
 For $s$ sufficiently small, we have   $2sc > 4s^2,$ so 
    \[ \limsup_{N \to \infty} \int_{0}^{1} R_3(s,N,\alpha) \,\mathrm{d}\alpha > 4s^2.\]
    By the reverse Fatou Lemma, 
    \[\limsup_{N \to \infty} \int_{0}^{1} R_3(s,N,\alpha) \,\mathrm{d}\alpha
    \ls  \int_{0}^{1} \limsup_{N \to \infty} R_3(s,N,\alpha) \,\mathrm{d}\alpha \]
    which implies that there must be a set $\Omega \subseteq [0,1]$ with positive Lebesgue measure such that for $\alpha \in \Omega,$
    \[\limsup_{N \to \infty} R_3(s,N,\alpha) > 4s^2.\]

\section{Proof of Theorem \ref{connection_to_higher_moments}}

In the proof of Theorem \ref{connection_to_higher_moments} we make use of the test functions $g^{(k)}_s:\mathbb{R}^{k-1}\to\mathbb{R}$ defined for every $k\gs 2$ and $s>0$ by
\begin{equation*}
    g^{(k)}_s(y_1,\ldots,y_{k-1}) := \Big\{s - \max_{1\ls i < k}\{y_i\}^{+} - \max_{1\ls i < k}\{-y_i\}^+ \Big\}^{+}.
\end{equation*}
Then  $g^{(k)}_s \in C_c(\mathbb{R}^{k-1})$. 
The importance of the test functions $g_s^{(k)}$ is seen from the following lemma. For convenience, when $s>0$ and $N\gs 1$ are fixed, given  an index $i \ls N$ we write $B_{i} = B\big(x_i,\dfrac{s}{2N}\big).$
 
 \begin{lem}\label{lem12} Let $s>0, N \gs 4s$ be fixed. Consider the points   $x_1, x_2,\ldots, x_k\in [0,1]$, where $2\ls k \ls N.$ Then 
 \begin{equation} \label{measure_of_intersection}
     \lambda\Big(\bigcap_{j=1}^{k}B_{{j}}\Big) = g_s^{(k)}\left(N(\!(x_{1} - x_{2})\!), \ldots ,N(\!(x_{1} - x_{k})\!) \right) . 
 \end{equation}
 \end{lem}
 \begin{proof}

Let $x_{i_1}$ and $x_{i_2}$ be the first and last point modulo $1$, that is, $(\!( x_i - x_{i_1})\!) \gs 0 $ and $(\!( x_{i_2} - x_i)\!) \gs 0 $ for all $1\ls i \ls k, i\neq i_1, i_2.$ Then 
 \begin{align*}\lambda\Big(\bigcap_{j=1}^{k}B_{{j}}\Big) &= \lambda( B_{i_1} \cap B_{i_2}) = \Big\{ \frac{s}{N} - \|x_{i_1}-x_{i_2}\|\Big\}^+ \\ & =  \left\{\frac{s}{N} - \max_{1\ls m,n \ls k} (\!(x_{m} - x_{n})\!)\right\}^{+}    \end{align*}
 and it remains to prove that 
 \begin{equation} \label{gs_comes_into_play}
     \left\{\frac{s}{N} - \max_{1\ls m,n \ls k} (\!(x_{m} - x_{n})\!)\right\}^{+}  = g_s^{(k)}\left(N(\!(x_{1} - x_{2})\!), \ldots ,N(\!(x_{1} - x_{k})\!) \right).
 \end{equation}

We only need to show this for points $x_1,\ldots,x_k$ such that $\|x_{1}-x_{\ell} \|\ls \frac{s}{N} \ls \frac{1}{4}$ for all $2\ls \ell \ls k$ \,--\, otherwise both sides of \eqref{gs_comes_into_play} are zero. For such points we have
 \begin{equation}\label{sum_of_norms}
 (\!( x_{\ell} - x_{j} )\!) =   (\!( x_{\ell} - x_{1} )\!) + (\!( x_{1} - x_{j} )\!)\qquad \text{for all } 1\ls \ell, j \ls k .
 \end{equation}

We further assume without loss of generality that the points $x_{2},\ldots, x_{k}$ are in increasing order, that is, $(\!( x_{n+1}-x_{n})\!)\gs 0 $ for $n=2,\ldots,k-1.$ We first consider the case in which $x_{1}$ is between $x_{2}$ and $x_{k},$ i.e. $(\!(x_{1}-x_{2})\!)>0$ and $(\!(x_{k}-x_{1})\!)>0.$  Then 
\begin{align*}
  \max_{1\ls m,n \ls k} (\!(x_{m} - x_{n})\!) &=  (\!( x_{k} - x_{2} )\!) \stackrel{(\ref{sum_of_norms})}{=}   (\!( x_{k} - x_{1} )\!) + (\!( x_{1} - x_{2} )\!) \\
  &=   \max_{2\ls n \ls k}\{ -(\!( x_{1} - x_{n} )\!) \}^+  + \max_{2\ls n \ls k}\{ (\!( x_{1} - x_{n} )\!)\}^+ ,
\end{align*}
which proves \eqref{gs_comes_into_play}. The proof is similar in all other cases regarding the relative position  of $x_1$ with respect to $x_2,\ldots, x_k.$ 
 \end{proof}

In order to employ the test functions $g_s^{(k)}$ to deduce information for Poissonian $k$\,--\,th order correlations, we first need to determine their integrals.

\begin{lemma} \label{integral_of_g} For any $k\gs 2$ and $s>0,$ we have
\begin{equation} \label{int_of_gs}
    \int_{\mathbb{R}^{k-1}} g^{(k)}_s(y_1,\ldots,y_{k-1}) \,\mathrm{d}y_1 \ldots\mathrm{d}y_{k-1} = s^{k}.
\end{equation} 
 \end{lemma}
 \begin{proof}
Clearly $\mathrm{supp}(g_s^{(k)}) \subseteq [-s,s]^{k-1}.$ We partition the set  $[-s,s]^{k-1}$ as follows: for each $\ell =0,1,\ldots, k-1$ and each subset $A\subseteq [k-1]$ with $|A|=\ell$ we define 
\[ D_\ell(A) = \{(y_1,\ldots, y_{k-1}) \in [-s,s]^{k-1}: y_i \gs 0 \iff i\in A  \}. \]
Then 
\[ \int g_s^{(k)}(y_1,\ldots,y_{k-1})\,\mathrm{d}y_1\ldots \mathrm{d}y_{k-1} = \sum_{\ell=0}^{k-1} \sum_{\substack{A\subseteq [k-1] \\ |A|=\ell}} \int_{D_{\ell}(A)} \hspace{-2mm} g_s^{(k)}(y_1,\ldots,y_{k-1})\,\mathrm{d}y_1\ldots \mathrm{d}y_{k-1}.  \]
Consider first the set $A_0 = \{1,\ldots, \ell \}.$ Then we can write 
\[ D_\ell(A_0) = \bigcup_{\substack{1\ls i \ls \ell\\ \ell < j <k}}D_{\ell}(A_0, i,j)   \]
where for each $1\ls i \ls \ell$ and $\ell < j < k$ we define
\[ D_{\ell}(A_0,i,j) = \left\{(y_1,\ldots, y_{k-1})\in D_{\ell}(A_0) : y_i=\max_{1\ls r \ls \ell }y_r, \, y_j = \min_{\ell < r <k}y_r \right\}.  \]
The sets $D_{\ell}(A_0,i,j)$ are almost pairwise disjoint (their intersections are sets of zero $(k-1)$\,--\,dimensional Lebesgue measure). On the set 
$D_{\ell}(A_0,1,k-1)$ we have for $1 \ls \ell < k-1$ that $$g_s^{(k-1)}(y_1,\ldots, y_{k-1}) =   s- y_1 + y_{k-1} .$$ Thus 
\begin{align*}
    \int\limits_{D_{\ell}(A_0,1,k-1)}\hspace{-4mm} g_s^{(k)} &= \int_0^s \int_{-(s-y_1)}^{0}\iint_{[0,y_1]^{\ell-1}\times [y_{k-1},0]^{k-\ell-2}} \hspace{-15mm}(s-y_1 + y_{k-1})\,\mathrm{d}y_2 \ldots \mathrm{d}y_{k-2}\mathrm{d}y_{k-1}\mathrm{d}y_1\\
    &= \int_0^s \int_{-(s-y_1)}^{0} (-1)^{k-\ell-2}y_1^{\ell-1} y_{k-1}^{k-\ell-2}(s-y_1+y_{k-1}) \,\mathrm{d}y_{k-1}\mathrm{d}y_1 \\
    &= \int_0^s (-1)^{k-\ell-2}x^{\ell-1}(s-x)\int_{-(s-x)}^0 y^{k-\ell-2}\,\mathrm{d}y \mathrm{d}x \\
    & \qquad + \int_0^s (-1)^{k-\ell-2}x^{\ell-1}\int_{-(s-x)}^0 y^{k-\ell-1} \,\mathrm{d}y  \\
    &= \Big(\frac{1}{k-\ell-1}  - \frac{1}{k-\ell}\Big)\int_0^s x^{\ell-1}(s-x)^{k-\ell}\,\mathrm{d}x.
\end{align*}
We now make use of the identity 
\[ \int_0^1 x^{n-1}(1-x)^{m-1}\,\mathrm{d}x = \frac1m \binom{m+n-1}{n-1}^{-1}, \qquad   m,n\gs 1 \]
(see for example \cite[p. 908, 910]{table}) to deduce that 
\[  \int\limits_{D_{\ell}(A_0,1,k-1)}\hspace{-4mm} g_s^{(k)} =  \frac{1}{(k-\ell)(k-\ell-1)\ell}    \binom{k}{k-\ell}^{-1} s^k.  \]
Using a symmetry argument, one sees that the integral of $g_s^{(k)}$ has the same value on any of the sets of the form $D_{\ell}(A_0,i,j)$. Since there exist $\ell (k-\ell-1)$ such sets, we have 
\begin{equation} \label{value_of_integral}
    \int_{D_{\ell}(A_0)}g_s^{(k)}\, = \,  \frac{1}{k-\ell}\binom{k}{k-\ell}^{-1} = \, \, \frac1k \binom{k-1}{\ell}^{-1} s^k. 
\end{equation}  
Another symmetry argument now shows that the value of the integral on any set $D_{\ell}(A)$ where $A\subset [k-1]$ has $\ell$ elements is the same as on the right hand side of \eqref{value_of_integral}. We have proved this for $\ell=1,\ldots, k-2$, but the same result also holds when $\ell=0$ or $k-1.$ Since there exist precisely $\binom{k-1}{\ell}$ such subsets of $[k-1],$ we have 
\[ \int_{D_\ell} g_s^{(k)} = \frac{s^k }{k} \cdot  \]
Finally, summing over the $k$ possible values of the index $\ell$ we deduce \eqref{int_of_gs}. 
\end{proof}

 Armed with Lemma \ref{lem12}, we can proceed to the proof of Theorem \ref{connection_to_higher_moments}.
 
\begin{proof}[Proof of Theorem \ref{connection_to_higher_moments}]
(i) Fix some $s>0$ and $N\gs 1$. 
A counting argument gives that for any $0\ls t \ls 1 $ such that $F(t,s,N)\gs k,$
\begin{equation}\label{counting_arg}F(t,s,N)(F(t,s,N)-1)\ldots (F(t,s,N)-(k-1)) = \hspace{-2mm}
 \sum_{\substack{i_1,\ldots ,i_k \ls N\\ {\rm distinct}}}\hspace{-2mm}
  \mathds{1}_{\left( \bigcap\limits_{j=1}^{k}B_{i_j}\right)}(t).\end{equation}
  Note that this equality also holds when $F(t,s,N) < k$: since $F$ is integer\,--\,valued, both sides of \eqref{counting_arg} are then equal to $0$.
Integrating with respect to $t$ we get  
\begin{align*}\nonumber I_k(s,N) &= \!\!\sum_{\substack{i_1,\ldots,i_k \ls N\\\mathrm{distinct}}}\lambda\Big( \bigcap_{j=1}^{k}B_{{i_j}}\Big) \stackrel{\eqref{measure_of_intersection}}{=}
%\nonumber&= \sum_{\substack{i_1,\ldots,i_k \ls N\\ \mathrm{distinct}}}\left\{\frac{s}{N} - \max_{1\ls m,n \ls k} (\!(x_{i_m} - x_{i_n})\!)\right\}^{+} \\
 \frac{1}{N}\sum_{\substack{i_1, \ldots,i_k \ls N\\ \mathrm{distinct}}}\!\!g_s^{(k)}\left(N(\!(x_{i_1} - x_{i_2})\!), \ldots ,N(\!(x_{i_1} - x_{i_k})\!) \right)  
\end{align*}

\noindent Since we assumed that $\xn$ has Poissonian $k$\,--\,th correlations, it follows from Lemma \ref{integral_of_g} that $ \lim\limits_{N\to\infty} I_k(s,N) = s^k.$

\noindent (ii) Using the well-known formula for Stirling numbers of the second kind
 
\[\sum_{j=1}^{k}c_{k,j}x(x-1)\cdots (x-(j-1)) = x^k\]
(see e.g. \cite[Chap. 5.3]{combinatorics}), we can write 
\begin{align}\nonumber
I_k^*(s,N) = \int_{0}^{1} F(t,s,N)^k \,\mathrm{d}t  =&    \sum_{\substack{i_1,\ldots,i_k\ls N\\{\rm distinct}}} \lambda\Big( \bigcap_{j=1}^{k}B_{x_{i_j}}\Big) + 
 \\\nonumber &+ c_{k,k-1} \sum_{\substack{i_1,\ldots,i_{k-1}\ls N\\ {\rm distinct}}} \lambda\Big( \bigcap_{j=1}^{k-1}B_{x_{i_j}}\Big) + \ldots
 \\& + c_{k,1} \sum_{i_1 \ls N} \lambda (B_{x_{i_1}}).  \nonumber 
 \end{align}
In view of \eqref{measure_of_intersection}, this implies that
\begin{align}
 \label{Ikstar}I_k^*(s,N)   &= R_k(g_s^{(k)},N) + c_{k,k-1} R_{k-1}(g_s^{(k-1)},N) + \ldots \\ & \qquad + c_{k,1}R_{2}(g_s^{(2)},N).\nonumber
\end{align}
The main term in \eqref{Ikstar} is $R_k(g_s^{(k)},N)= I_k(s,N),$ and we have shown in (i) that $I_k(s,N)\to s^k, N\to \infty.$  For $\ell < k$, we see that $g_s^{(\ell)} \ls s\mathds{1}_{[-s,s]^{\ell-1}}$ which implies by monotonicity of $R_\ell(\cdot, N)$ that
 \[  R_{\ell}(g_s^{(\ell)},N)  \ls sR_{\ell}(s,N).  \]
By Proposition \ref{HaukePrinciple}, since $\xn$ has Poissonian $k$\,--\,th correlations we have  
\[ \limsup_{N\to\infty}R_{\ell}(s,N) = \mathcal{O}_k(s^{\ell-1}), \qquad s\to \infty.  \]
Combining \eqref{Ikstar} with these remarks, we see that $$\limsup\limits_{N\to\infty}I_k^*(s,N) = s^k + \mathcal{O}_k(s^{k-1}), \quad s\to\infty.$$
(iii)  Since $\xn$ has Poissonian $\ell$\,--\,correlations for all $2\ls \ell \ls k$, by definition $\lim\limits_{N\to\infty}R_{\ell}(g_s^{(\ell)},N)=s^{\ell}.$  Therefore in that case, \eqref{Ikstar} implies that
\[ \lim_{N\to\infty} I_k^*(s,N) = s^k + c_{k,k-1}s^{k-1} + \ldots + c_{k,1}s. \]
\end{proof}

 \section*{Appendix A: Equivalent Definitions of Poissonian Correlations}
 
We discuss the different definitions of Poissonian $k$\,--\,th order correlations appearing in the literature. \par
The reader who is already familiar with the notion of Poissonian correlations might compare the definition given in \eqref{poisson} with another common definition, where in the correlation function $R_k(g,N)$ the differences $(\!(x_{i_1}-x_{i_2})\!), (\!(x_{i_2}-x_{i_3})\!),\ldots, (\!(x_{i_{k-1}}-x_{i_k})\!)$ appear  instead of the differences $(\!(x_{i_1}-x_{i_2})\!),(\!(x_{i_1}-x_{i_3})\!)\ldots (\!(x_{i_1}-x_{i_k})\!)$ as in \eqref{kthcorrelationfunction}. 
This is the definition used, for example, in the papers \cite{k,kr} that deal with the $k$\,--\,th level correlations of quadratic residues modulo some integer $Q\gs 1$. \par Here we explain that these two definitions are equivalent. \\

\noindent {\bf Proposition A. } {\it  Let $\xn \subseteq [0,1]$ be a sequence.
The following are equivalent: \newline
%\begin{itemize}
(i) The sequence $\xn$ has Poissonian $k$\,--\,th order correlations. \newline
(ii) For all test functions $g \in C_c(\mathbb{R}^{k-1})$ we have
\[\lim_{N \to \infty}\frac{1}{N}\!\!\sum_{\substack{i_1,\ldots,i_k \ls N\\ \mathrm{ distinct}}}\hspace{-2mm} g\left(N(\!(x_{i_1}-x_{i_2})\!),N(\!(x_{i_2}-x_{i_3})\!),\ldots,N(\!(x_{i_{k-1}}-x_{i_k})\!)\right) = \int\limits_{\mathbb{R}^{k-1}} g(x) \,\mathrm{d}x.\]
(iii) For all rectangles  $B = [a_1,b_1]\times [a_2,b_2] \times \ldots \times [a_{k-1},b_{k-1}],\; b_i > a_i,\; 1  \ls i \ls k-1,$ we have
\[\lim_{N \to \infty} R_k(\mathds{1}_{B},N) = \lambda(B)\]
where $\lambda$ denotes the $(k-1)$\,--\,dimensional Lebesgue measure.   }

\begin{proof} The proof of the equivalence of (i) and (iii) uses a standard approximation argument from analysis and is omitted. 
    We  show (i) $\Rightarrow$ (ii), the direction (ii) $\Rightarrow$ (i) can be proven in a similar fashion. Let $g \in C_c(\mathbb{R}^{k-1})$ be an arbitrary test function, and define $f: \mathbb{R}^{k-1} \to \mathbb{R}$ via
    \[f(x_1,x_2,\ldots,x_{k-1})  = g(x_1,x_2-x_1,x_3-x_2,\ldots,x_{k-1}-x_{k-2}).\]  
This definition implies that $f \in C_c(\mathbb{R}^{k-1})$,
   \begin{equation}\label{same_integral}
    \int_{\mathbb{R}^{k-1}} f(x) \,\mathrm{d}x = \int_{\mathbb{R}^{k-1}} g(x) \,\mathrm{d}x,
    \end{equation}
 and furthermore 
    \begin{equation}\label{g_f_translate}g(x_1,x_2,\ldots,x_{k-1}) = f\left(x_1,x_1+x_2,\ldots,x_1+x_2 + \ldots +x_{k-1}\right).\end{equation}
 Now when $N\gs 1$ is so large that $ \mathrm{supp}(f) \subseteq \left[\frac{-N}{2k},\frac{N}{2k}\right]^{k-1}$ 
    %\begin{equation}\label{support_contained} \mathrm{supp}(f) \subseteq \left[\frac{-N}{2k},\frac{N}{2k}\right]^{k-1}.\end{equation}
    and in addition 
    \[f\left(N(\!(x_{i_1}-x_{i_2})\!),N(\!(x_{i_2}-x_{i_3})\!),\ldots,N(\!(x_{i_{k-1}}-x_{i_k})\!)\right) \neq 0,\]
    then for all $\ell \ls k$ we have $\lVert x_{i_{\ell-1}}-x_{i_{\ell}} \rVert \ls \dfrac{1}{2k}.$ Hence 
    \begin{equation*}\lVert x_{i_1} - x_{i_2}\rVert + \lVert x_{i_2} - x_{i_3}\rVert + \ldots + \lVert x_{i_{\ell-1}} - x_{i_\ell} \rVert \ls \frac{1}{2}
    \end{equation*}    
and thus
  \begin{equation}\label{signed_distance_cond}(\!(x_{i_1} - x_{i_2})\!) + (\!(x_{i_2} - x_{i_3})\!) + \ldots + (\!(x_{i_{\ell-1}} - x_{i_\ell})\!) = (\!(x_{i_1} - x_{i_\ell})\!).
    \end{equation}
Using these considerations, we get
    \begin{align*}
    \frac{1}{N}& \sum_{\substack{i_1,\ldots,i_k \ls N\\ \text{ distinct}}}\hspace{-2mm} g\left(N(\!(x_{i_1}-x_{i_2})\!),N(\!(x_{i_2}-x_{i_3})\!), \ldots,N(\!(x_{i_{k-1}}-x_{i_k})\!)\right) \\ 
    & \stackrel{\eqref{g_f_translate}}=  \frac{1}{N}\sum_{\substack{i_1,\ldots, i_k \ls N\\ \text{ distinct}}}\hspace{-2mm}
    f\big(N(\!(x_{i_1}-x_{i_2})\!), N\!\left((\!(x_{i_1}-x_{i_2})\!) + (\!(x_{i_2}-x_{i_3})\!)\right),\ldots\big) 
    \\   &\stackrel{\eqref{signed_distance_cond}}=  \frac{1}{N}\sum_{\substack{i_1,\ldots,i_k \ls N\\ \text{ distinct}}}\hspace{-2mm}f\left(N(\!(x_{i_1}-x_{i_2})\!),N(\!(x_{i_1}-x_{i_3})\!),\ldots,N(\!(x_{i_1}-x_{i_k})\!)\right) = R_k(f,N).
    \end{align*}
    Combining this with \eqref{same_integral}, we see that under the hypothesis that the sequence has Poissonian $k$\,--\,th correlations, statement (ii) is true. As noted, the implication \\(ii) $\Rightarrow$ (i) can be shown in a similar way.
    \end{proof}

At this point, we should also mention that for the specific case $k=2$,  sequences are usually defined to have \textit{Poissonian pair correlations} when 
\begin{equation} \label{PPCdef}
\lim_{N\to \infty} R_2(s,N) = 2s \qquad \text{ for all } s>0, 
\end{equation}
where $R_2(s,N)$ is the correlation function as in \eqref{kthcorrelationfunction_old_def}. We have already explained why a sequence with Poissonian pair correlations automatically satisfies \eqref{PPCdef}, but it turns out that condition \eqref{PPCdef} is actually equivalent to \eqref{poisson} when $k=2$. Indeed, assume $\xn$ is a sequence such that \eqref{PPCdef} holds. For the test function $\one_{[0,s]}$ (that is, the characteristic function of the interval $[0,s]$) we have 
\begin{align}\label{0s2}
R_2(\mathds{1}_{[0,s]},N) 
&= \frac{1}{2} R_2(s,N) + \frac{1}{N}\#\big\{i \neq j \ls N: x_i = x_j\big\}.
\end{align}
%\todoMH[inline]{Define $\mathds{1}$ already here?}
Since
\[0 \ls \frac{1}{N}\#\{ i \neq j \ls N: x_i = x_j\} \ls R_2(\varepsilon,N) \qquad \text{ for any } \varepsilon > 0, \]
the assumption on $\xn$ implies that the rightmost term in \eqref{0s2} tends to $0$. By this, we deduce that $ \lim_{N \to \infty} R_2(\one_{[0,s]},N) = s  \text{ for all } s>0 $ and it follows  that for any $b > a$ we have $\lim_{N\to \infty} R_2(\mathds{1}_{[a,b]},N) = b-a.$ In view of Proposition A (iii), relation \eqref{poisson} holds for $k=2$. \par We also note that this equivalence cannot be  generalized for correlations of higher orders.  The reason is that  when $k\gs 3$, we cannot employ the symmetry argument to deduce  a generalization of \eqref{0s2}, i.e. for the rectangle $B = [0,s_1]\times   \ldots \times [0,s_{k-1}]$, it does not hold in general that
        \begin{equation*}\label{non_equivalence}
        R_k(\mathds{1}_B,N) = \Big(\frac{1}{2}\Big)^{k-1}R_k(s_1,\ldots,s_{k-1},N).\end{equation*}

 \section*{Appendix B: The $k$\,--\,th order correlations of random sequences are almost surely Poissonian}
  We establish a fact that was alluded to in the introduction: whenever $(Y_n)_{n\in\bN}$ is a sequence of independent and uniformly distributed random variables in $[0,1]$, then for all $k\gs 2$ the sequence $(Y_n(\om))_{n\in\bN}$ almost surely has Poissonian correlations of $k$\,--\,th order. The method of proof we use is a standard mean\,--\,variance argument that appears very often in the relevant literature. We refer the reader to \cite{hinrichs} for a proof in higher dimensions when $k=2$.

Indeed, given such a sequence $(Y_n)_{n\in\bN}$ and $a_r<b_r \, \, (1\ls r < k)$, let 
\[ R_k(\omega,N) = \frac{1}{N}\# \left\{ \parbox{7em}{$ i_1,\ldots,i_k \leqslant N $ \\  $i_j \neq i_\ell \,\forall j \neq \ell$}\hspace{-5mm}: \frac{a_r}{N} \ls Y_{i_1}(\omega)-Y_{i_{r+1}}(\om) \ls \frac{b_r}{N}\quad (1\ls r < k) \right\}  . \]
According to Proposition A it suffices to show that $R_k(\omega,N)\to (b_1-a_1) \cdots (b_{k-1}-a_{k-1})$ as $N\to \infty$ almost surely.

We  proceed to the calculation of the expectation and variance of $R_k(\cdot, N)$  viewed as a random variable on $[0,1].$ This will be done in several steps. In the   estimates that follow, the implicit constants in the asymptotic notations depend on the scales $a_1,\ldots, b_{k-1}.$
\newline \textbf{ Step 1:}  We claim that for any $m\gs 1$ and for any distinct indices $i_1, i_2, \ldots, i_m, \gs 1$ the differences  
\[ \D_1 = Y_{i_1} - Y_{i_2},\quad \D_2 = Y_{i_1} - Y_{i_3},\quad \ldots \quad \D_{m+1} = Y_{i_1} - Y_{i_{m+2}}  \] are independent.
For convenience, we only show this when $m=1$ and for the random variables
$$\D_1=Y_1-Y_2 \quad \text{ and } \quad \D_2= Y_1-Y_3 $$  (the proof is similar for any other choice of indices and for any $m\gs 2$).
%(differences of the form $Y_i-Y_j$ and $Y_k-Y_\ell$ with all indices distinct are automatically independent by the hypothesis). 
Writing $f_{X}$ for the probability density function of a random variable $X:[0,1]\to \bR$, we need to show that
\begin{equation}\label{independence} 
f_{\D_1,\D_2}(x_1,x_2) = f_{\D_1}(x_1)f_{\D_2}(x_2) \quad \text{ for all } x_1,x_2\in [0,1].\end{equation} Note that in our context, addition and subtraction of the random variables is always understood modulo $1$. Thus for all $i\in\bN$ we have $f_{Y_i}(y)=1,$  $0\ls y \ls 1$ and the density functions of the differences $\D_1, \D_2$ are 
\[f_{\D_1}(\delta) = \int_0^1 f_{Y_1}(y)f_{Y_2}(\delta+y)\mathrm{d}y = 1  \]
and similarly $f_{\D_2}(\delta)= 1,$ for all $\delta \in [0,1].$ 
We start from the left\,--\,hand side of \eqref{independence}: the theorem of total probability and the independence of
$\{Y_1, Y_2,Y_3\}$ imply that 
\begin{align*}
    f_{\D_1,\D_2}(x_1,x_2) &= \int_{0}^{1}f_{\D_1,\D_2 \vert Y_1}(x_1,x_2|y) \mathrm{d}y = \int_0^1 \frac{f_{\D_1,\D_2,Y_1}(x_1,x_2,y)}{f_{Y_1}(y)}\mathrm{d}y \\ 
    &= \int_{0}^{1} f_{\D_1,\D_2,Y_1}(x_1,x_2,y)\mathrm{d}y = \int_0^1 f_{y-Y_2}(x_1)f_{y-Y_3}(x_2)\mathrm{d}y =1.
\end{align*}
\textbf{Step 2:} We now claim that whenever each of the sets $I = \{i_1,\ldots,i_m\}$ and  $J = \{j_1,\ldots, j_n\}$ consists of pairwise distinct indices and $\#(I \cap J) \ls 1$, then the differences $Y_{i_1}-Y_{i_2}, Y_{i_1}-Y_{i_3}, \ldots, Y_{i_1}-Y_{i_m}, Y_{j_1}-Y_{j_2},\ldots, Y_{j_1}-Y_{j_{n}}$
are independent random variables.
It follows by the previous step that the random variables $Y_{i_1}-Y_{i_2}, Y_{i_1}-Y_{i_3}, \ldots, Y_{i_1}-Y_{i_m}$ are independent, and so are $Y_{j_1}-Y_{j_2},\ldots ,Y_{j_1}-Y_{j_{n}}$ .
If the sets $I,J$ are disjoint, all of these random variables are immediately independent, so we deal with the case when $\#(I\cap J) =1.$ We need to distinguish the following cases: \newline
Case 1: $i_1 = j_1$. In this case, independence follows immediately from Step $1$. \newline
Case 2: $i_1 = j_{\ell}$ or $j_1 = i_{\ell}$ with $\ell \gs 2$: Then we can assume without loss of generality that $i_1 = j_2$. The random variables we are looking at are
\[Y_{j_1}-Y_{j_n},Y_{j_1} - Y_{j_{n-1}},\ldots,Y_{j_1} - Y_{j_2} = Y_{j_1} - Y_{i_1},
Y_{i_1} - Y_{i_2},\ldots,Y_{i_1} - Y_{i_m},
\]
and these are independent by an argument similar to the one used in Step $1.$ \newline
Case 3: $i_r = j_{\ell}$ with $r,\ell \gs 2$: Then we assume without loss of generality $i_2 = j_2$ and we are in a similar situation as in Case 2. \newline \textbf{Step 3:} Given $k$ distinct indices $i_1,\ldots, i_k \ls N,$ we write \[ X_{i_1,\ldots,i_{k}} =  \one_{[\frac{a_1}{N} \ls Y_{i_1}-Y_{i_2}\ls \frac{b_1}{N}]}\cdots \one_{[ \frac{a_{k-1}}{N}\ls Y_{i_1}-Y_{i_{k}}\ls \frac{b_{k-1}}{N}]}. \]
We calculate the covariance $\text{Cov}(X_{i_1,\ldots,i_k},X_{j_1,\ldots,j_k}),$ where $i_1,\ldots,i_k \ls N$ and $j_1,\ldots, j_k\ls N$ both consist of distinct indices. As seen above, when 
$\#\{i_1,\ldots,i_k\} \cap \{j_1,\ldots,j_k\} \ls 1$ the variables $X_{i_1,\ldots,i_{k}}$ and $X_{j_1,\ldots,j_{k}}$ are independent and thus the covariance in question is $0$. We need to bound the covariance of $X_{i_1,\ldots,i_{k}},X_{j_1,\ldots,j_{k}}$ when
\begin{equation}\label{intersecting_indices}
\#\{i_1,\ldots,i_k\} \cap \{j_1,\ldots,j_k\} = \ell \gs 2.
\end{equation}
If $(i_1,\ldots,i_k),(j_1,\ldots,j_k)$ fulfill \eqref{intersecting_indices}, then
\begin{align} 
\text{Cov}(X_{i_1,\ldots,i_k},X_{j_1,\ldots,j_k}) &=  \mathbb{E}\big[ \big(X_{i_1,\ldots,i_k} - \mathbb{E}[X_{i_1,\ldots,i_k}] \big)  \big(X_{j_1,\ldots,j_k} - \mathbb{E}[X_{j_1,\ldots,j_k}] \big)\big] \nonumber \\
&\ls   \mathbb{E}\big[ X_{i_1,\ldots,i_k} \cdot X_{j_1,\ldots,j_k} \big] \nonumber \\
  &=  \mathbb{E}\Big[ \one_{[\frac{a_1}{N} \ls Y_{i_1}-Y_{i_2}\ls \frac{b_1}{N}]}\cdots \one_{[ \frac{a_{k-1}}{N}\ls Y_{j_1}-Y_{j_{k}}\ls \frac{b_{k-1}}{N}]}\Big]. \label{the_integral_to_be_estimated}
\end{align}
To find an upper bound for the right\,--\,hand side, we take a maximal subset $\tilde{J} \subset \{j_2,\ldots,j_k\}$ such that 
$\#\{i_1,\ldots,i_k\} \cap (\{j_1\} \cup \tilde{J}) \ls 1$. By the discussion in Step $2$, this gives rise to $2(k-1) - (\ell - 1)$ independent random variables among the factors in the integral in \eqref{the_integral_to_be_estimated}, each of them having expectation $\mathcal{O}\left(\frac{1}{N}\right)$. The remaining differences appearing in \eqref{the_integral_to_be_estimated} might not be independent, but can be trivially bounded from above by $1$ since they are characteristic functions. We therefore conclude that 
\[ \text{Cov}(X_{i_1,\ldots,i_k},X_{j_1,\ldots,j_k}) = \mathcal{O}\Big( \frac{1}{N^{2(k-1) - (\ell -1)}} \Big)\cdot \]
\textbf{Step 4: }We fix a value of $\ell\gs 2$ and we estimate 
\[ \sum_{ \substack{i_1,\ldots, i_k\ls N\\ j_1,\ldots, j_{k}\ls N\\ \mathrm{distinct}
\\ \#\{i_1,\ldots,i_k\} \cap \{j_1,\ldots,j_k\} = \ell}} \hspace{-8mm} \text{Cov}(X_{i_1,...,i_k},X_{j_1,...,j_k}). \] 
We count how many choices of the sets $\{i_1,\ldots, i_k\}$ and $\{j_1, \ldots, j_k \}$ there exist such that \eqref{intersecting_indices} holds. From a fixed set $\{i_1,\ldots,i_k\}$ we can chose $\ell$ elements that will be in the intersection with $\{j_1, \ldots, j_k\}$ in $\binom{k}{\ell}$ ways. The remaining $k-\ell$ elements of $\{j_1, \ldots, j_k\}$ can be chosen in $\mathcal{O}(N^{k-\ell})$ ways. Since the set $\{i_1,\ldots, i_k\}$ can be chosen in $\mathcal{O}_k(N^k)$ ways, we deduce that the different choices of the sets $\{i_1,\ldots, i_k\}$ and $\{j_1, \ldots, j_k \} $ with precisely $\ell$ common elements is $\mathcal{O}_{k,\ell}(N^{2k-\ell})$. Combining with Step $3$, we obtain
\begin{equation}\label{covariance_bound}
\sum_{ \substack{i_1,\ldots, i_k\ls N\\ j_1,\ldots, j_{k}\ls N\\ \mathrm{distinct}
\\ \#\{i_1,\ldots,i_k\} \cap \{j_1,\ldots,j_k\} = \ell}} \hspace{-8mm} \text{Cov}(X_{i_1,...,i_k},X_{j_1,...,j_k})
= \mathcal{O}_{k,\ell}\big( N^{2k - \ell} \frac{1}{N^{2(k-1) - (\ell -1)}}\big) = \mathcal{O}_{k,\ell}( N ).
\end{equation}
\textbf{Step 5: } We are finally in place to calculate the expectation and variance of $R_k(\cdot, N).$  The expectation is
\begin{align*}
\mathbb{E}[R_k(\cdot,N)] &= \frac{1}{N}\sum_{\substack{i_1,\ldots, i_{k}\ls N\\ \mathrm{distinct}}} \mathbb{E}\big[ \one_{[\frac{a_1}{N}\ls Y_{i_1}-Y_{i_2}\ls \frac{b_1}{N}]}\cdots \one_{[\frac{a_{k-1}}{N}\ls Y_{i_1}-Y_{i_{k}}\ls \frac{b_{k-1}}{N}]}\big] \\
&= \frac{1}{N}\sum_{\substack{i_1,\ldots, i_{k}\ls N\\ \mathrm{distinct}}} \mathbb{E}\big[ \one_{[\frac{a_1}{N} \ls Y_{i_1}-Y_{i_2}\ls \frac{b_1}{N}]}\big] \cdots \mathbb{E}\big[\one_{[ \frac{a_{k-1}}{N}\ls Y_{i_1}-Y_{i_{k}}\ls \frac{b_{k-1}}{N}]}\big] \\
&=  (a_1-b_1) \cdots (a_{k-1}-b_{k-1})  + \mathcal{O}\Big(\frac{1}{N}\Big), \qquad N\to \infty 
\end{align*}
because $(Y_n)_{n\in\bN}$ are independent (here we used the result from Step $1$). For the variance we have 
\begin{align}\label{variance_estimation}
  \text{Var}[R_k(\cdot,N)] &=   \frac{1}{N^2}\sum_{\ell = 0}^{k}\hspace{-6mm} 
  \sum_{ \substack{i_1,\ldots, i_k\ls N\\ j_1,\ldots, j_{k}\ls N\\ \mathrm{distinct}, \;\,(i_1,...,i_k) \neq (j_1,...,j_k)
\\ \#\{i_1,\ldots,i_k\} \cap \{j_1,\ldots,j_k\} = \ell}}\hspace{-8mm} \text{Cov} (X_{i_1,...,i_k},X_{j_1,...,j_k}) \nonumber
\\+ &\frac{1}{N^2}\sum_{\substack{i_1,\ldots, i_{k}\ls N\\ \mathrm{distinct}}} \mathrm{Var}[ \one_{[\frac{a_1}{N}\ls Y_{i_1}-Y_{i_2}\ls \frac{b_1}{N}]}\cdots \one_{[\frac{a_{k-1}}{N}\ls Y_{i_1}-Y_{i_{k}}\ls \frac{b_{k-1}}{N}]}] .
\end{align}
In the first sum appearing in \eqref{variance_estimation}, the terms corresponding to $\ell=0$ or $1$ are equal to $0.$ Applying \eqref{covariance_bound} for the terms corresponding to $2 \ls \ell \ls k$ in \eqref{variance_estimation}, we obtain 
\begin{align*}\text{Var}[R_k(\cdot,N)] = &\frac{1}{N^2}\sum_{\substack{i_1,\ldots, i_{k}\ls N\\ \mathrm{distinct}}} \mathrm{Var}[ \one_{[\frac{a_1}{N}\ls Y_{i_1}-Y_{i_2}\ls \frac{b_1}{N}]}\cdots \one_{[\frac{a_{k-1}}{N}\ls Y_{i_1}-Y_{i_{k}}\ls \frac{b_{k-1}}{N}]}] +
\mathcal{O}(\frac{1}{N}).
\end{align*} 
Since 
\begin{align*}
\frac{1}{N^2}&\sum_{\substack{i_1,\ldots, i_{k}\ls N\\ \mathrm{distinct}}} \mathrm{Var}[ \one_{[\frac{a_1}{N}\ls Y_{i_1}-Y_{i_2}\ls \frac{b_1}{N}]}\cdots \one_{[\frac{a_{k-1}}{N}\ls Y_{i_1}-Y_{i_{k}}\ls \frac{b_{k-1}}{N}]}]\\
&= \frac{1}{N^2}\sum_{\substack{i_1,\ldots, i_{k}\ls N\\ \mathrm{distinct}}}\Big( \bE[\one_{[ \frac{a_1}{N}\ls Y_{i_1}-Y_{i_2}\ls \frac{b_1}{N}]}\cdots \one_{[\frac{a_{k-1}}{N} \ls Y_{i_1}-Y_{i_{k}}\ls \frac{b_{k-1}}{N}]}]- \\[-3ex]  & \qquad \qquad \qquad \qquad  \bE[\one_{[ \frac{a_1}{N}\ls Y_{i_1}-Y_{i_2}\ls \frac{b_1}{N}]}\cdots \one_{[ \frac{a_{k-1}}{N}\ls Y_{i_1}-Y_{i_{k}}\ls \frac{b_{k-1}}{N}]}]^2\Big) \\[1ex]
&= \frac{1}{N^2}\sum_{\substack{i_1,\ldots, i_{k}\ls N\\ \mathrm{distinct}}}\Big(\frac{a_1-b_1}{N}\cdot \ldots \cdot \frac{a_{k-1}-b_{k-1}}{N} - \prod_{r=1}^{k-1}\Big( \frac{a_r-b_r}{N} \Big)^2 \Big) \\
& =  \mathcal{O}_k\big(\frac{1}{N}\big),
\end{align*}
we deduce that 
\[ \text{Var}[R_k(\cdot,N)] = \mathcal{O}_k\big(\frac{1}{N}\big), \qquad N\to \infty.\]
Consider now the sequence $N_m = \lfloor m^{1+\gamma}\rfloor, m\gs 1$ where $\gamma>0$. Fix $\varepsilon>0$ and let  $A_N = \{\om \in [0,1] : |R_k(\om,N)-\mathbb{E}[R_k(\cdot, N)] | \gs \varepsilon\}, N\gs 1.$ By Chebyshev's inequality, the Lebesgue measure of $A_N$ satisfies $ \lambda(A_N) = \mathcal{O}(1/N), N\to \infty. $ Then the Borel\,--\,Cantelli lemma implies that $$ \lambda\big(\limsup\limits_{m\to\infty}A_{N_m}\big)=0,$$ and since $\varepsilon>0$ was arbitrarily chosen, we conclude that for almost all $\om\in [0,1]$ we have $\lim\limits_{m\to\infty}R_k(\om,N_m) =  (b_1-a_1) \cdots  (b_{k-1}-a_{k-1}).$ 
%We have proved that for a fixed choice of the scales $a_r< b_{r}\, (1\ls r <k)$ the random variable $R_k$ converges almost surely to the desired limit $ (b_1-a_1) \cdots (b_{k-1}-a_{k-1}).$ 
The fact that $R_k(\cdot,N_m)$ converges almost surely to this limit for any choice of the scalars follows by an intersection over a countable dense set of $(a_1,\ldots, b_{k-1}).$   
\par It remains to prove that for the same values of $\om$, $R_k(\om,N)$ will converge to the value $(b_1-a_1) \cdots  (b_{k-1}-a_{k-1})$ for any choice of $a_1,\ldots,b_{k-1}$. This will follow from the fact that  when $N_m \ls N < N_{m+1}$ we have 
\begin{align*}
\frac{1}{N} \sum_{\substack{i_1,\ldots, i_{k}\ls N_m\\ \mathrm{distinct}}}   \hspace{-2mm}\one_{[ \frac{a_1}{N_m}\frac{N_m}{N} \ls Y_{i_1}-Y_{i_2}\ls \frac{b_1}{N_m}\frac{N_m}{N}]}(\om)\cdots \one_{[ \frac{a_{k-1}}{N_m}\frac{N_m}{N}\ls Y_{i_1}-Y_{i_{k}}\ls \frac{b_{k-1}}{N_m}\frac{N_m}{N}]}(\om) \ls R_k(\om,N) & \\ 
\ls \frac{1}{N}\!\!\sum_{\substack{i_1,\ldots, i_{k}\ls N_{m+1}\\ \mathrm{distinct}}}   \hspace{-4mm}\one_{[\frac{a_1N_{m+1}}{N_{m+1}N}  \ls Y_{i_1}-Y_{i_2}\ls \frac{b_1N_{m+1}}{N_{m+1}N} ]}(\om)\cdots \one_{[ \frac{a_{k-1}N_{m+1}}{N_{m+1}N} \ls Y_{i_1}-Y_{i_{k}}\ls \frac{b_{k-1}N_{m+1}}{N_{m+1}N} ]}(\om) &
\end{align*}
and the fact that $N_m/ N_{m+1} \to 1$ as $m\to \infty$.

\subsection*{Acknowledgements} We would like to thank Professor C. Aistleitner for suggesting this direction of research. We also thank the anonymous referee for many valuable remarks.

\end{document}